\documentclass[11pt]{amsart}
\usepackage{amssymb,amsmath,amsfonts}
\usepackage{amsthm}
\usepackage{cite}
\usepackage[textwidth=160mm,textheight=220mm]{geometry}

\newtheorem{lemma}{Lemma}[section]
\newtheorem{theo}{Theorem}[section]
\newtheorem{cor}{Corollary}[section]
\newtheorem{prop}{Proposition}[section]

\theoremstyle{definition}
\newtheorem{defn}{Definition}[section]
\newtheorem{rem}{Remark}[section]

\numberwithin{equation}{section}

\newcommand{\tin}{t\in [a,b]}
\newcommand{\cnul}{C_0\big(\RR;\RR\big)}
\newcommand{\cnulid}{C_0\big(I;D\big)}
\newcommand{\cnulok}{C_0\big(\RR;[0,\kappa]\big)}
\newcommand{\cnulrp}{C_0\big(\RR;\RR_+\big)}
\newcommand{\cnultnr}{C_0\big([t_0,+\infty);\RR\big)}
\newcommand{\cloc}{C_{loc}\big(\RR;\RR\big)}
\newcommand{\clocrprp}{C_{loc}\big(\RR_+;\RR_+\big)}
\newcommand{\clocrnn}{C_{loc}\big(\RR;(0,+\infty)\big)}
\newcommand{\cabr}{C\big([a,b];\RR\big)}
\newcommand{\cabrp}{C\big([a,b];\RR_+\big)}
\newcommand{\cantr}{C\big([a_n,t_0];\RR\big)}
\newcommand{\cantrp}{C\big([a_n,t_0];\RR_+\big)}
\newcommand{\catr}{C\big([a,t_0];\RR\big)}
\newcommand{\catrp}{C\big([a,t_0];\RR_+\big)}
\newcommand{\ctbr}{C\big([t_0,b];\RR\big)}
\newcommand{\acloc}{AC_{loc}\big((-\infty,t_0];(0,+\infty)\big)}
\newcommand{\acloctau}{AC_{loc}\big((-\infty,\tau];(0,+\infty)\big)}
\newcommand{\aclocrp}{AC_{loc}\big((-\infty,t_0];\RR_+\big)}
\newcommand{\acloctrp}{AC_{loc}\big((-\infty,\tau];\RR_+\big)}
\newcommand{\acloctaur}{AC_{loc}\big((-\infty,\tau];\RR\big)}
\newcommand{\aclocr}{AC_{loc}\big(\RR;\RR\big)}
\newcommand{\aclocrpos}{AC_{loc}\big(\RR;(0,+\infty)\big)}
\newcommand{\aclocid}{AC_{loc}\big(I;D\big)}
\newcommand{\acabr}{AC\big([a,b];\RR\big)}
\newcommand{\acatr}{AC\big([a,t_0];\RR\big)}
\newcommand{\actbr}{AC\big([t_0,b];\RR\big)}
\newcommand{\acataur}{AC\big([a,\tau];\RR\big)}
\newcommand{\acabd}{AC\big([a,b];D\big)}
\newcommand{\acabrp}{AC\big([a,b];\RR_+\big)}
\newcommand{\acatrp}{AC\big([a,t_0];\RR_+\big)}
\newcommand{\acaton}{AC\big([a,t_0];(0,+\infty)\big)}
\newcommand{\acatauon}{AC\big([a,\tau];(0,+\infty)\big)}
\newcommand{\lloc}{L_{loc}\big(\RR;\RR\big)}
\newcommand{\llocrp}{L_{loc}\big(\RR;\RR_+\big)}
\newcommand{\lloctrp}{L_{loc}\big((-\infty,t_0];\RR_+\big)}
\newcommand{\llocirp}{L_{loc}\big(I;\RR_+\big)}
\newcommand{\linf}{L^{+\infty}\big(\RR;\RR\big)}
\newcommand{\lantrp}{L\big([a_n,t_0];\RR_+\big)}
\newcommand{\labrp}{L\big([a,b];\RR_+\big)}
\newcommand{\latrp}{L\big([a,t_0];\RR_+\big)}
\newcommand{\ltbrp}{L\big([t_0,b];\RR_+\big)}
\renewcommand{\for}{\qquad\mbox{for~}\,}
\newcommand{\mfor}{\quad\mbox{for~}\,}
\newcommand{\forae}{\qquad\mbox{for~a.~e.~}\,}
\newcommand{\mforae}{\quad\mbox{for~a.~e.~}\,}
\newcommand{\casfor}{\mbox{for~}\,}
\newcommand{\casforae}{\mbox{for~a.~e.~}\,}
\newcommand{\casif}{\mbox{if~}\,}
\newcommand{\saba}{\mathcal{S}_{ab}(a)}
\newcommand{\sabac}{\mathcal{S}'_{ab}(a)}
\newcommand{\satac}{\mathcal{S}'_{at_0}(a)}
\newcommand{\satauac}{\mathcal{S}'_{a\tau}(a)}
\newcommand{\sata}{\mathcal{S}_{at_0}(a)}
\newcommand{\sabb}{\mathcal{S}_{ab}(b)}
\newcommand{\satt}{\mathcal{S}_{at_0}(t_0)}
\newcommand{\santt}{\mathcal{S}_{a_nt_0}(t_0)}
\newcommand{\satautau}{\mathcal{S}_{a\tau}(\tau)}
\newcommand{\sataua}{\mathcal{S}_{a\tau}(a)}
\newcommand{\aab}{\mathcal{A}(a,b)}
\newcommand{\vtau}{V_{\tau}}
\newcommand{\vtn}{V_{t_0}}
\newcommand{\vtd}{V_{t_{\delta}}}
\newcommand{\vb}{V_b}
\newcommand{\atnb}{\mathcal{A}(t_0,b)}
\newcommand{\pplus}{\mathcal{P}^+}
\newcommand{\ptnplus}{\mathcal{P}^+_{t_0}}
\newcommand{\ptplus}{\mathcal{P}^+_{\tau}}
\newcommand{\klocirp}{K_{loc}\big(I\times\RR_+;\RR_+\big)}
\newcommand{\kloctnrp}{K_{loc}\big([t_0,+\infty)\times\RR_+;\RR_+\big)}
\newcommand{\kabrp}{K\big([a,b]\times\RR_+;\RR_+\big)}
\newcommand{\ktnbrp}{K\big([t_0,b]\times\RR_+;\RR_+\big)}
\newcommand{\sgn}{\operatorname{sgn}}
\newcommand{\esssup}{\operatorname{ess~sup}}
\newcommand{\essinf}{\operatorname{ess~inf}}

\def\RR{\mathbb{R}}
\def\NN{\mathbb{N}}

\begin{document}

\title[Existence and Properties of Semi-Bounded Solutions]
{Existence and Properties of Semi-Bounded Global Solutions to the Functional Differential
  Equation with Volterra's Type Operators on the Real Line}

\author{Maitere Aguerrea}
\address{Facultad de Ciencias B\'asicas, Universidad Cat\'olica del
  Maule, Casilla 617, Talca, Chile}
\email{maguerrea@ucm.cl}
\thanks{Supported by FONDECYT/INICIACION/ N$^o$ 11121457.}

\author{Robert Hakl}
\address{Institute of Mathematics, Academy of Sciences of the Czech
  Republic, \v{Z}i\v{z}kova 22, 616 62 Brno, Czech Republic}
\email{hakl@ipm.cz}
\thanks{Supported by RVO: 67985840}

\begin{abstract}
Consider the equation
$$
u'(t)=\ell_0(u)(t)-\ell_1(u)(t)+f(u)(t)\forae t\in\RR
$$
where $\ell_i:\cloc\to\lloc$ $(i=0,1)$ are linear positive continuous operators and
$f:\cloc\to\lloc$ is a continuous operator satisfying the local
Carath\'eodory conditions. The efficient conditions 
guaranteeing the existence of a global solution, which is bounded and
non-negative in the neighbourhood of $-\infty$, to the equation
considered are established provided $\ell_0$, $\ell_1$, and $f$ are
Volterra's type operators. The existence of a solution which is positive on
the whole real line is discussed, as well. Furthermore, the asymptotic properties of
such solutions are studied in the neighbourhood of
$-\infty$. The results are applied to certain models appearing in
natural sciences. 
\end{abstract}


\maketitle

\section{Introduction}
\label{sec1}

Many models in natural sciences are based on the idea that the
derivative at a certain moment of time depends not only on the present
state but on some of the previous states. However, in spite of the
fact that the history of delay differential equations goes back to the
beginning of the 20th century (see, e.g., the works of Vito Volterra),
or even more back in time, the systematic study of such types of
equations started only in the beginning of the 1950s.

The main purpose of the present paper is to study the existence and
asymptotic properties of a global solution (i.e., defined on the whole
real line) to the scalar functional differential equation
\begin{equation}
\label{1.1}
u'(t)=\ell_0(u)(t)-\ell_1(u)(t)+f(u)(t).
\end{equation}
Here, $\ell_i:\cloc\to\lloc$ $(i=0,1)$ are linear continuous operators
which are positive, i.e., they transform non-negative functions into
the set of non-negative functions, and $f:\cloc\to\lloc$ is a
continuous operator satisfying the local Carath\'eodory conditions,
i.e., for every $r>0$ there exists $q_r\in\llocrp$ such that 
$$
|f(v)(t)|\leq q_r(t)\forae t\in\RR
$$
whenever
$$
\sup\big\{|v(t)|:t\in\RR\big\}\leq r.
$$
Together with the equation \eqref{1.1} consider the condition
\begin{equation}
\label{1.2}
u(t_0)=c
\end{equation}
with $t_0,c\in\RR$.

By a global solution to the equation \eqref{1.1} we understand a
function $u:\RR\to\RR$ which is absolutely continuous on every compact
interval and satisfies \eqref{1.1} for almost every $t\in\RR$.
Effective sufficient conditions for the existence of a global solution to the
problem \eqref{1.1}, \eqref{1.2} are established in the paper. More
precisely, we are interested in the study of existence of global
positive semi-bounded (i.e., bounded in the neighbourhood of $-\infty$)
solutions $u:\RR\to \RR$ to the problem \eqref{1.1}, \eqref{1.2}. 

The study of the geometric property and the existence of solutions to this class of
problems was motivated by the open problem concerned with degenerate
scalar reaction-diffusion equations with delay  
$$
\phi_t(t,x) = \phi_{xx}(t,x)- \phi(t,x) + G(\phi(t-r,x)), \qquad x \in \RR,\quad
r>0,  
$$
and the existence of positive semi-wavefront solutions $\phi(t,x) =
u(x+ct)$, $u(-\infty)=0$, when $G\in C^1(\RR_+,\RR_+)$,
$G'(0)=1,$ and $0$ and $\kappa>0$ are the only two solutions of
$G(s)=s$ (degenerate monostable case). When we do not consider
diffusion, we obtain the following equation with the boundary
condition  
\begin{equation}
\label{m1}
\begin{cases}
cu'(t)=-u(t)+G(u(t-cr)),&\\
u(-\infty)=0, &
\end{cases}
\end{equation}
with degenerate monostable nonlinearity $G$. The existence problem for
\eqref{m1} and their generalizations have been investigated in several
papers and approached by means of different methods and almost always
assuming the generate condition $G'(0)>1$.  It is worthwhile
mentioning that in the proofs of existence this condition is essential
and cannot be omitted or weakened within the framework (see
\cite{ft,map,tat} and references therein).  
 
In the case when $r=0$, without delay, only a few theoretical studies have
considered the important problem when $G'(0)=1$, i.e., the degenerate
case (see \cite{bn,zlh}). These works show that the assumption
$G'(0)>1$ is not necessary to obtain the existence and the geometric
properties of travelling solutions of a nonlocal dispersal problem or
parabolic equations. Motivated for these investigations we have
developed a more general theory that can be applied to the problem
with delay \eqref{m1} and hence to complete or to improve the research
on existence problems done so far.  

The results of the paper can be also applied to the scalar delay
logistic equation of the form
\begin{equation}
\label{log.eq}
u'(t)=u(t)F(t,u_t)
\end{equation}
describing the population growth (see Section~\ref{sec6}). The
asymptotic properties at $+\infty$ of such kinds
of models where studied e.g. in \cite{dom,f1} (see also references therein). 
For more model differential equations used in natural sciences which
can be rewritten in the form of \eqref{1.1} we recommend \cite{dom}
and references therein.

In this way, the main results of our work can be reformulated for a
particular case of the equation \eqref{1.1}, for the equation  with
argument deviation of the form  
\begin{equation}
\label{1.3}
u'(t)=p_0(t)u(\mu_0(t))-p_1(t)u(\mu_1(t))+h(t,u(t),u(\nu(t))),
\end{equation}
where $p_i\in\llocrp$, $\mu_i,\nu:\RR\to\RR$ are locally essentially
bounded measurable functions, $\mu_i(t)\leq t$, $\nu(t)\leq t$ for
almost every $t\in\RR$ $(i=0,1)$, and $h:\RR^3\to \RR$ is a function
satisfying local Carath\'eodory conditions, i.e., $h(\cdot,x,y):\RR\to\RR$
is measurable for every $x,y\in\RR$, $h(t,\cdot,\cdot):\RR^2\to\RR$ is
continuous for almost every $t\in\RR$, and for every $r>0$ there exists
$q_r\in\llocrp$ such that 
$$
|h(t,x,y)|\leq q_r(t)\forae t\in\RR,\quad |x|+|y|\leq r.
$$
Note that, when $p_0(t)=p_1(t)=1/c$,  $\mu_0(t)=\nu(t)=t-cr$,
$\mu_1(t)=t$ and $h(t,u(t),u(\nu(t)))=\big[G(u(t-cr))-u(t-cr)\big]/c$, equation
\eqref{1.3} reduces to the model equation \eqref{m1}.

The paper is organized as follows. Basic notation used in the paper
can be found just after this introduction in
Section~\ref{sec1}. The main results dealing with the existence and
positivity of a global semi-bounded solution to the problem
\eqref{1.1}, \eqref{1.2}, as well as the conditions guaranteeing that
such a solution has a limit at $-\infty$ equal to zero, are
established in Section~\ref{sec2}. The results of Section~\ref{sec2}
are reformulated for the particular case of \eqref{1.1}---the equation
\eqref{1.3}---in Section~\ref{sec3}. Sections~\ref{sec4} and
\ref{sec5} are devoted to the auxiliary propositions and proofs of the
main results, respectively. Applications of the obtained results to
the model problem \eqref{m1} and generalized logistic equation
\eqref{log.eq} can be found in Section~\ref{sec6}. 

\subsection{Basic notation}

The following notation is used throughout the paper:

$\NN$ is a set of all natural numbers.

$\RR$ is a set of all real numbers; $\RR_+=[0,+\infty)$;
$\RR^2=\RR\times\RR$; $\RR^3=\RR^2\times\RR$.

$\cabr$ is a Banach space of continuous functions $u:[a,b]\to \RR$ with the norm 
$$
\|u\|_{a,b}=\max\big\{|u(t)|:\tin\big\}.
$$

$\cabrp=\big\{u\in\cabr:u(t)\geq 0\mbox{~for~}\tin\big\}$. 

$\acabd$, where $D\subseteq\RR$, is a set of absolutely continuous
functions $u:[a,b]\to D$.

$\labrp$ is a set of Lebesgue-integrable functions $p:[a,b]\to \RR_+$. 

$\cloc$ is a space of continuous functions $u:\RR\to\RR$ with the
topology of uniform convergence on every compact interval.

If $u\in\cloc$ then $u(-\infty)$, resp. $u(+\infty)$, stands for a
limit (finite or infinite) of $u$ at $-\infty$, resp. $+\infty$, if
such a limit exists.

$\cnul$ is a Banach space of bounded continuous functions
$u:\RR\to\RR$ with the norm
$$
\|u\|=\sup\big\{|u(t)|:t\in \RR\big\}.
$$

$\cnulid$, where $I\subseteq\RR$, $D\subseteq\RR$, is a set of bounded
continuous functions $u:I\to D$. 

$\aclocid$, where $I\subseteq\RR$, $D\subseteq \RR$, is a
set of functions $u:I\to D$ which are absolutely continuous on every
compact interval contained in $I$.

$\lloc$ is a space of locally Lebesgue-integrable functions $p:\RR\to\RR$ with the
topology of convergence in the mean on every compact interval.

$\llocirp$, where $I\subseteq \RR$, is a set of functions
$p:I\to\RR_+$ which are Lebesgue-integrable on every compact interval
contained in $I$.

$\linf$ is a Banach space of essentially bounded measurable functions $p:\RR\to\RR$ with the
norm
$$
\|p\|_{\infty}=\esssup\big\{|p(t)|:t\in\RR\big\}.
$$

$\kabrp$ is the Carath\'eodory class, i.e., the set of functions
$q:[a,b]\times\RR_+\to\RR_+$ such that $q(\cdot,x):[a,b]\to\RR_+$ is
measurable for any $x\in\RR_+$, $q(t,\cdot):\RR_+\to\RR_+$ is
continuous for almost all $\tin$, and
$$
\sup\big\{q(\cdot,x):x\in D\big\}\in\labrp
$$
for any compact set $D\subset\RR_+$.

$\klocirp$, where $I\subseteq \RR$, is a set of functions
$q:I\times\RR_+\to\RR_+$ such that $q\in\kabrp$ for every
$[a,b]\subset I$. 

Let $I\subseteq\RR$ be a closed interval, $T:\cloc\to\lloc$ be a
continuous operator, and let $u:I\to\RR$ be a continuous
function. Then we put
$$
T(u)(t)\stackrel{def}{=}T(\vartheta(u))(t)\forae t\in\RR
$$
where
\begin{equation}
\label{1.4}
\vartheta(u)(t)=
\begin{cases}
u(\inf I)&\casfor t<\inf I\quad\casif\inf I>-\infty,\\
u(t)&\casfor t\in I,\\
u(\sup I)&\casfor t>\sup I\quad\casif\sup I<+\infty.
\end{cases}
\end{equation}

$\ptplus$, where $\tau\in\RR$, is a set of all linear continuous operators $\ell:\cloc\to\lloc$
such that
$$
\ell(u)(t)\geq 0\forae t\leq \tau
$$
whenever $u\in\acloctrp$ is a nondecreasing function.

$\vtau$, where $\tau\in\RR$, is a set of all continuous operators
$T:\cloc\to\lloc$ such that, for arbitrary $\zeta\leq\tau$, the equality
$$
T(u)(t)=T(v)(t)\forae t\leq \zeta
$$
holds whenever $u,v\in\cloc$ are such that
$$
u(t)=v(t)\for t\leq \zeta.
$$

$\Sigma$ is a set of all continuous functions $\sigma:\RR\to\RR$
satisfying $\sigma(t)\leq t$ for $t\in\RR$.

$\vtau(\sigma)$, where $\tau\in\RR$ and $\sigma\in\Sigma$, is a set
of all continuous operators $T:\cloc\to\lloc$ with a memory $\sigma$
on $(-\infty,\tau]$, i.e., for almost every $t\leq\tau$, the equality 
$$
T(u)(t)=T(v)(t)
$$
holds provided $u,v\in\cloc$ are such that 
$$
u(s)=v(s)\for s\in[\sigma(t),t].
$$

$V=\bigcap\limits_{\tau\in\RR}\vtau$,
$\pplus=\bigcap\limits_{\tau\in\RR}\ptplus$.

Let $I\subseteq\RR$ be a closed interval. By a solution to the
equation \eqref{1.1} on the interval $I$ we understand a function
$u:I\to\RR$ which is absolutely continuous on every compact interval
contained in $I$ and satisfies \eqref{1.1} almost everywhere on
$I$.\footnote{Remind that by $\ell_0(u)$, $\ell_1(u)$, and $f(u)$ we
  understand $\ell_0(\vartheta(u))$, $\ell_1(\vartheta(u))$, and
  $f(\vartheta(u))$, respectively, where $\vartheta$ is given by
  \eqref{1.4}.} If, moreover, $t_0\in I$ then by a solution to the
problem \eqref{1.1}, \eqref{1.2} on $I$ we understand a solution $u$
to \eqref{1.1} on $I$ satisfying \eqref{1.2}. If $I=\RR$ then we speak
about a global solution.

\section{Main Results}
\label{sec2}

\subsection{Existence theorems}

\begin{theo}
\label{t2.1}
Let $\ell_0,\ell_1,f\in V$, $\sigma\in\Sigma$,
\begin{gather}
\label{2.1}
\ell_0\in \vtn(\sigma),\\
\label{2.2}
f(0)(t)=0\forae t\leq t_0,
\end{gather}
and let there exist $\kappa>0$ such that 
\begin{equation}
\label{2.3}
f(v)(t)\geq 0 \forae t\leq t_0,\quad v\in\cnulok
\end{equation}
and
\begin{multline}
\label{2.3.5}
f(v)(t)\sgn v(t)\leq q(t,\|v\|) \forae t\geq t_0,\quad
v\in\cnul,\\ 
0\leq v(t_0)\leq \kappa,\quad -\kappa\leq v(t)\leq \kappa\mfor t\leq t_0,
\end{multline}
where $q\in\kloctnrp$ is nondecreasing in the second argument and
satisfies
\begin{equation}
\label{3.5}
\lim_{x\to +\infty}\frac{1}{x}\int_{t_0}^b q(s,x)ds=0
\end{equation}
for every $b>t_0$.
Let, moreover, there exist $\gamma\in\acloc$ such that 
\begin{gather}
\label{2.4}
\gamma'(t)\leq -\ell_1(\gamma)(t)\forae t\leq t_0,\\
\label{2.5}
\ell_0(1)(t)\geq \frac{\ell_1(\gamma)(t)}{\gamma(t)}\forae t\leq t_0,\\
\label{2.6}
\sup\left\{\int_{\sigma(t)}^t\frac{\ell_1(\gamma)(s)}{\gamma(s)}ds:t\leq t_0\right\}<+\infty.
\end{gather}
Then, for every $c\in \left[0,\kappa e^{-M_{\sigma}}\right]$, where
\begin{equation}
\label{2.7}
M_{\sigma}=\sup\left\{\int_{\sigma(t)}^t\frac{\ell_1(\gamma)(s)}{\gamma(s)}ds:t\leq t_0\right\},
\end{equation}
there exists a global solution $u$ to the problem \eqref{1.1}, \eqref{1.2} satisfying 
\begin{equation}
\label{2.8}
0\leq u(t)\leq \kappa\for t\leq t_0.
\end{equation}
\end{theo}

\begin{rem}
\label{r2.1}
Note that the solution $u$, the existence of which is guaranteed by
Theorem~\ref{t2.1}, has also a finite limit $u(-\infty)$ according to
Theorem~\ref{t2.10} formulated below.
\end{rem}

\begin{theo}
\label{t2.2}
Let $\ell_0,\ell_1,f\in V$ be such that $\ell_0-\ell_1\in\ptnplus$,
\eqref{2.2} holds, and let there exist $\kappa>0$ such that \eqref{2.3} and
\eqref{2.3.5} are fulfilled, where $q\in\kloctnrp$ is nondecreasing in
the second argument and satisfies \eqref{3.5} for every $b>t_0$.
Then, for every $c\in[0,\kappa]$ there exists a global solution $u$ to the problem
\eqref{1.1}, \eqref{1.2} satisfying \eqref{2.8} and
\begin{equation}
\label{2.9}
u'(t)\geq 0\forae t\leq t_0.
\end{equation}
\end{theo}

\begin{rem}
\label{r2.2}
Note that the solution $u$, the existence of which is guaranteed by
Theorem~\ref{t2.2}, has also a finite limit $u(-\infty)$ because $u$
is a bounded nondecreasing function in the neighbourhood of $-\infty$.
\end{rem}

\begin{theo}
\label{t2.3}
Let all the assumptions of Theorem~\ref{t2.1} be fulfilled. Let,
moreover, there exist $g\in\llocrp$ and a continuous nondecreasing function
$h_0:(0,+\infty)\to (0,+\infty)$ such that
\begin{equation}
\label{2.10}
f(v)(t)\leq g(t)h_0(\|v\|)\forae t\leq t_0, \quad v\in\cnulok,\quad
v\not\equiv 0
\end{equation}
and
\begin{equation}
\label{2.11}
\lim_{x\to 0_+}\int_x^1\frac{ds}{h_0(s)}=+\infty.
\end{equation}
Then, for every $c\in \left(0,\kappa e^{-M_{\sigma}}\right]$
with $M_{\sigma}$ given by \eqref{2.7}, there exists a global solution $u$
to the problem \eqref{1.1}, \eqref{1.2} satisfying  
\begin{equation}
\label{2.12}
0<u(t)\leq \kappa\for t\leq t_0.
\end{equation}
\end{theo}

\begin{rem}
\label{r2.1.5}
Note that the solution $u$, the existence of which is guaranteed by
Theorem~\ref{t2.3}, has also a finite limit $u(-\infty)$ according to
Theorem~\ref{t2.10} formulated below.
\end{rem}

\begin{theo}
\label{t2.4}
Let all the assumptions of Theorem~\ref{t2.2} be fulfilled. Let,
moreover, there exist $g\in\llocrp$ and a continuous nondecreasing function
$h_0:(0,+\infty)\to (0,+\infty)$ such that \eqref{2.10} and \eqref{2.11} hold.
Then, for every $c\in (0,\kappa]$ there exists a global solution $u$ to the
problem \eqref{1.1}, \eqref{1.2} satisfying \eqref{2.9} and \eqref{2.12}.
\end{theo}

\begin{rem}
\label{r2.2.5}
Note that the solution $u$, the existence of which is guaranteed by
Theorem~\ref{t2.4}, has also a finite limit $u(-\infty)$ because $u$
is a bounded nondecreasing function in the neighbourhood of $-\infty$.
\end{rem}

\begin{theo}
\label{t2.1.r}
Let $\ell_0,\ell_1,f\in V$, $\sigma\in\Sigma$, \eqref{2.1} and
\eqref{2.2} hold, and let there exist $\kappa>0$ such that \eqref{2.3}
and \eqref{2.3.5} are fulfilled where $q\in\kloctnrp$ is nondecreasing in the second argument and
satisfies \eqref{3.5} for every $b>t_0$. Let, moreover,
\begin{equation}
\label{2.3.5.r}
\ell_0(v)(t)+f(v)(t)\geq 0 \mforae t\geq t_0,\quad
v\in\cnulrp, \quad v(t)\leq \kappa\mfor t\leq t_0.
\end{equation}
Let, in addition, there exist $\gamma\in\aclocrpos$ such that 
\begin{equation}
\label{2.4.r}
\gamma'(t)\leq -\ell_1(\gamma)(t)\forae t\in\RR,
\end{equation}
and \eqref{2.5} and \eqref{2.6} are satisfied.
Then, for every $c\in \left(0,\kappa e^{-M_{\sigma}}\right]$, where
$M_{\sigma}$ is given by \eqref{2.7},
there exists a global solution $u$ to the problem \eqref{1.1}, \eqref{1.2}
satisfying \eqref{2.8} and
\begin{equation}
\label{2.8.r}
u(t)>0 \for t\geq t_0.
\end{equation}
\end{theo}

\begin{rem}
\label{r2.1.r}
Note that the solution $u$, the existence of which is guaranteed by
Theorem~\ref{t2.1.r}, has also a finite limit $u(-\infty)$ according to
Theorem~\ref{t2.10} formulated below.
\end{rem}

\begin{theo}
\label{t2.2.r}
Let $\ell_0,\ell_1,f\in V$ be such that $\ell_0-\ell_1\in\pplus$,
\eqref{2.2} holds, and let there exist $\kappa>0$ such that \eqref{2.3} and
\begin{equation}
\label{2.3.5.rr}
0\leq f(v)(t)\leq q(t,\|v\|) \mforae t\geq t_0,\quad
v\in\cnulrp, \quad v(t)\leq \kappa\mfor t\leq t_0
\end{equation}
are fulfilled, where $q\in\kloctnrp$ is nondecreasing in
the second argument and satisfies \eqref{3.5} for every $b>t_0$.
Then, for every $c\in[0,\kappa]$ there exists a global solution $u$ to the problem
\eqref{1.1}, \eqref{1.2} satisfying \eqref{2.8} and
\begin{equation}
\label{2.9.r}
u'(t)\geq 0\forae t\in\RR.
\end{equation}
\end{theo}

\begin{rem}
\label{r2.2.1.r}
Obviously, if $c>0$ in Theorem~\ref{t2.2.r} then \eqref{2.9.r} implies \eqref{2.8.r}.
\end{rem}

\begin{rem}
\label{r2.2.r}
Note that the solution $u$, the existence of which is guaranteed by
Theorem~\ref{t2.2.r}, has also a finite limit $u(-\infty)$ because $u$
is a bounded nondecreasing function in the neighbourhood of $-\infty$.
\end{rem}

Theorems~\ref{t2.3}--\ref{t2.2.r} together with Remark~\ref{r2.2.1.r}
imply the following results dealing with the existence of global solutions to
\eqref{1.1}, \eqref{1.2} which are positive on the whole real
line.

\begin{cor}
\label{c1.1}
Let all the assumptions of Theorem~\ref{t2.1.r} be fulfilled. Let,
moreover, there exist $g\in\llocrp$ and a continuous nondecreasing function
$h_0:(0,+\infty)\to (0,+\infty)$ such that \eqref{2.10} and \eqref{2.11} hold.
Then, for every $c\in \left(0,\kappa e^{-M_{\sigma}}\right]$
with $M_{\sigma}$ given by \eqref{2.7}, there exists a positive global solution $u$
to the problem \eqref{1.1}, \eqref{1.2} satisfying
\begin{equation}
\label{2.12.r}
u(t)\leq \kappa\for t\leq t_0.
\end{equation}
\end{cor}

\begin{cor}
\label{c1.2}
Let all the assumptions of Theorem~\ref{t2.2.r} be fulfilled. Let,
moreover, there exist $g\in\llocrp$ and a continuous nondecreasing function
$h_0:(0,+\infty)\to (0,+\infty)$ such that \eqref{2.10} and \eqref{2.11} hold.
Then, for every $c\in (0,\kappa]$ there exists a positive global solution $u$ to the
problem \eqref{1.1}, \eqref{1.2} satisfying \eqref{2.9.r} and \eqref{2.12.r}.
\end{cor}

\begin{rem}
\label{r2.99}
Note that, according to Remarks~\ref{r2.1.r} and \ref{r2.2.r}, the
solution $u$, the existence of which is guaranteed by 
Corollary~\ref{c1.1}, resp. Corollary~\ref{c1.2}, has also a finite limit $u(-\infty)$.
\end{rem}

Corollaries~\ref{c1.1} and \ref{c1.2} are direct consequences of
Theorems~\ref{t2.3}--\ref{t2.2.r} and
Remark~\ref{r2.2.1.r}. Therefore, their proofs are omitted.

\subsection{Properties of solutions}

\begin{theo}
\label{t2.9}
Let $\ell_1\in\vtn$ and let there exist $\kappa>0$ such that 
\begin{equation}
\label{2.23}
\ell_0(v)(t)+f(v)(t)\geq 0 \forae t\leq t_0,\quad v\in\cnulok.
\end{equation}
Let, moreover, there exist $\gamma\in\acloc$ such that \eqref{2.4} is
fulfilled and
\begin{equation}
\label{2.24}
\gamma(-\infty)<+\infty.
\end{equation}
Then every solution $u$ to \eqref{1.1} on $(-\infty,t_0]$ satisfying
\eqref{2.8} has a finite limit $u(-\infty)$.
\end{theo}

\begin{rem}
\label{r2.6}
Note that the conditions \eqref{2.4}, \eqref{2.24}, and the inclusion $\ell_1\in\vtn$
imply that the function $\ell_1(1)$ is integrable in the neighbourhood of $-\infty$,
i.e., 
\begin{equation}
\label{2.25}
\lim_{t\to -\infty}\int_t^{t_0}\ell_1(1)(s)ds<+\infty.
\end{equation}
Indeed, from \eqref{2.4} it follows that $\gamma$ is a nonincreasing
function, and thus
\begin{equation}
\label{2.26}
\gamma'(t)\leq -\ell_1(1)(t)\gamma(t)\forae t\leq t_0
\end{equation}
(see Lemma~\ref{l2.1} below with $\ell=\ell_1$, $\alpha=1$,
$\beta=-\vartheta(\gamma)$ and $\vartheta$ given by \eqref{1.4}). Now \eqref{2.26} yields 
\begin{equation}
\label{2.27}
0<\gamma(t_0)\leq\gamma(t)
\exp\left(-\int_t^{t_0}\ell_1(1)(s)ds\right)\for t\leq t_0
\end{equation}
and so \eqref{2.25} holds provided \eqref{2.24} is fulfilled.

On the other hand, from \eqref{2.27} it follows that if 
$$
\lim_{t\to -\infty}\int_t^{t_0}\ell_1(1)(s)ds=+\infty,
$$
then necessarily $\gamma(-\infty)=+\infty$. If the latter occurs,
one can apply the following theorem.
\end{rem}

\begin{theo}
\label{t2.10}
Let $\ell_1\in\vtn$, $\sigma\in\Sigma$, \eqref{2.1} hold,   
and let there exist $\kappa>0$ such that \eqref{2.3} is fulfilled.
Let, moreover, there exist $\gamma\in\acloc$ such that
\eqref{2.4}--\eqref{2.6} are satisfied.
Then every solution $u$ to \eqref{1.1} on $(-\infty,t_0]$ satisfying \eqref{2.8}
has a finite limit $u(-\infty)$.
\end{theo}

To formulate our next result we introduce
\begin{defn}
\label{d2.1}
Let $\omega\in\Sigma$, and let $\tau\in\RR$, $\kappa>0$, and $c\in (0,\kappa)$ be
constants. An operator $T:\cloc\to\lloc$ belongs to the
set $\mathcal{O}_{\tau}(\omega,\kappa,c)$ if
$$
\limsup_{t\to -\infty}\int_{\omega(t)}^t T(u)(s)ds>0
$$
whenever $u\in\cnulok$ is such that
$$
u(\tau)=c,
$$
there exists a finite limit $u(-\infty)\leq c$, and 
\begin{gather*}
0<u(-\infty)\leq u(t)\for t\leq
\tau\quad\casif u(-\infty)<c,\\
u(t)=c\for t\leq \tau\quad\casif u(-\infty)=c.
\end{gather*}
\end{defn}

\begin{theo}
\label{t2.11}
Let $\ell_0,\ell_1\in\vtn$, 
\begin{equation}
\label{2.29}
\ell_0(1)(t)\geq \ell_1(1)(t)\forae t\leq t_0,
\end{equation}
and let there exist $\kappa>0$ such that \eqref{2.3} holds. Let, moreover,
there exist a function $\gamma\in\acloc$ such that
\eqref{2.4} and \eqref{2.24} are fulfilled. 
Assume, further, that $c\in(0,\kappa)$ and $u$ is a solution to the problem
\eqref{1.1}, \eqref{1.2} on $(-\infty,t_0]$ satisfying \eqref{2.8} and
having a limit $u(-\infty)$. If either
\begin{equation}
\label{2.30}
\lim_{t\to -\infty}\int_t^{t_0} \ell_0(1)(s)ds=+\infty
\end{equation}
or there exists $\omega\in\Sigma$ such that
\begin{equation}
\label{2.31}
f\in\mathcal{O}_{t_0}(\omega,\kappa,c)
\end{equation}
then 
\begin{equation}
\label{2.32}
u(-\infty)=0.
\end{equation}
\end{theo}

\begin{theo}
\label{t2.12}
Let $\ell_0,\ell_1\in\vtn$, \eqref{2.29} hold,
and let there exist $\kappa>0$ such that \eqref{2.3} is satisfied. Let,
moreover, there exist $\gamma\in\acloc$ satisfying \eqref{2.4}. 
Assume, further, that $c\in(0,\kappa)$ and $u$ is a solution to the problem
\eqref{1.1}, \eqref{1.2} on $(-\infty,t_0]$ satisfying \eqref{2.8} and having a limit $u(-\infty)$.
If there exists $\omega\in\Sigma$ such that
\begin{equation}
\label{2.33}
\sup\left\{\int_{\omega(t)}^t\ell_1(1)(s)ds:t\leq t_0\right\}<+\infty,
\end{equation}
and either
\begin{equation}
\label{2.34}
\limsup_{t\to -\infty}\int_{\omega(t)}^t \big[\ell_0(1)(s)-\ell_1(1)(s)\big]ds>0
\end{equation}
or \eqref{2.31} holds then \eqref{2.32} is fulfilled. 
\end{theo}

Now we formulate a sufficient condition for the
inclusion \eqref{2.31}.

\begin{prop}
\label{p2.1}
Let $f\in\vtn$, $\omega\in\Sigma$, and let there exist $\kappa>0$,
$g\in\lloc$, and a continuous operator $h_1:\cnul\to\linf$ such that 
\begin{equation}
\label{2.35}
f(v)(t)\geq g(t)h_1(v)(t)\forae t\leq t_0,\quad v\in\cnulok.
\end{equation}
Let, moreover, 
\begin{equation}
\label{2.36}
\sup\left\{\int_{\omega(t)}^t |g(s)|ds:t\leq t_0\right\}<+\infty
\end{equation}
and
\begin{equation}
\label{2.37}
\limsup_{t\to -\infty}\int_{\omega(t)}^t g(s)h_1(x)(s)ds>0
\end{equation}
for every constant function $x:\RR\to (0,\kappa)$. Then \eqref{2.31} holds
for every $c\in(0,\kappa)$.
\end{prop}

\begin{cor}
\label{c2.1}
Let $\ell_0,\ell_1,f\in\vtn$, \eqref{2.29} hold,
and let there exist $\kappa>0$, $g\in\llocrp$, and a continuous operator
$h_1:\cnul\to\linf$ such that 
\begin{equation}
\label{2.217}
h_1(v)(t)\geq 0\forae t\leq t_0,\quad v\in\cnulok
\end{equation}
and \eqref{2.35} is satisfied. Let, moreover, there exist
$\gamma\in\acloc$ satisfying \eqref{2.4} and \eqref{2.24}. 
Assume, further, that $c\in(0,\kappa)$ and either \eqref{2.30} holds or 
\begin{gather}
\label{2.38}
\lim_{t\to -\infty}\int_t^{t_0}g(s)ds=+\infty,\\
\label{2.39}
\lim_{t\to -\infty}\essinf\big\{h_1(x)(s): s\leq t\big\}>0
\end{gather} 
for every constant function $x:\RR\to (0,\kappa)$. Then every solution $u$
to the problem \eqref{1.1}, \eqref{1.2} on $(-\infty,t_0]$ satisfying \eqref{2.8} 
has a finite limit $u(-\infty)$ and \eqref{2.32} holds.
\end{cor}

\begin{cor}
\label{c2.2}
Let $\ell_1,f\in\vtn$, $\sigma\in\Sigma$, \eqref{2.1} hold,
and let there exist $\kappa>0$, $g\in\llocrp$, and a continuous operator
$h_1:\cnul\to\linf$ such that \eqref{2.35} and \eqref{2.217} are satisfied. Let, moreover, there exist
$\gamma\in\acloc$ satisfying \eqref{2.4}--\eqref{2.6}. 
Assume, further, that $c\in(0,\kappa)$ and there exists $\omega\in\Sigma$ such that \eqref{2.33} is fulfilled
and either \eqref{2.34} holds or \eqref{2.36} and \eqref{2.37}
for every constant function $x:\RR\to (0,\kappa)$ are satisfied. Then every solution $u$
to the problem \eqref{1.1}, \eqref{1.2} on $(-\infty,t_0]$ satisfying
\eqref{2.8} has a finite limit $u(-\infty)$ and \eqref{2.32} holds.
\end{cor}

\section{Equation with deviating arguments}
\label{sec3}

Now we establish assertions dealing with the equation \eqref{1.3}.

\subsection{Existence theorems}

\begin{theo}
\label{t2.5}
Let there exist $\kappa>0$ such that 
\begin{gather}
\label{2.13}
h(t,x,y)\geq 0\forae t\leq t_0,\quad x,y\in[0,\kappa],\\
\label{2.13.5}
h(t,x,y)\sgn x\leq q(t,|x|+|y|) \forae t\geq t_0,\quad
x,y\in\RR,
\end{gather}
where $q\in\kloctnrp$ is nondecreasing in the second argument and
satisfies \eqref{3.5} for every $b>t_0$.
Let, moreover, 
\begin{gather}
\label{2.14}
h(t,0,0)=0\forae t\leq t_0,\\
\label{2.15.r}
\mu_0(t)\leq t,\qquad \mu_1(t)\leq t,\qquad \nu(t)\leq t
\forae t\in\RR,\\
\label{2.16}
\int_{\mu_1(t)}^t p_1(s)ds\leq \frac{1}{e}\forae t\leq t_0,\\
\label{2.17}
p_0(t)\geq p_1(t)\exp\left(e\int_{\mu_1(t)}^t p_1(s)ds\right)\forae
t\leq t_0,\\
\label{2.18}
\esssup\left\{\int_{\mu_0(t)}^t p_1(s)ds:t\leq t_0\right\}<+\infty.
\end{gather}
Then, for every $c\in \left[0,\kappa e^{-M_{\mu}}\right)$, where
\begin{equation}
\label{2.19}
M_{\mu}=\esssup\left\{\int_{\mu_0(t)}^t p_1(s)\exp\left(e\int_{\mu_1(s)}^sp_1(\xi)d\xi\right)ds:t\leq t_0\right\},
\end{equation}
there exists a global solution $u$ to the problem \eqref{1.3}, \eqref{1.2}
satisfying \eqref{2.8}.
\end{theo}

\begin{rem}
\label{r2.3}
Condition \eqref{2.17} in Theorem~\ref{t2.5} can be weakened to 
\begin{equation}
\label{100}
p_0(t)\geq
p_1(t)\exp\left(\lambda\int_{\mu_1(t)}^tp_1(s)ds\right)\forae t\leq t_0
\end{equation}
where $\lambda\in[1,e]$ satisfies 
\begin{equation}
\label{101}
\lambda=e^{\lambda p^*},\qquad p^*=\esssup\left\{\int_{\mu_1(t)}^tp_1(s)ds:t\leq t_0\right\}.
\end{equation}
Obviously, in that case the number $M_{\mu}$ can also be improved in
an appropriate sense.
\end{rem}

\begin{rem}
\label{r2.4}
Note that the solution $u$, the existence of which is guaranteed by
Theorem~\ref{t2.5}, has also a finite limit $u(-\infty)$ (see
Theorem~\ref{t2.14}).
\end{rem}

\begin{theo}
\label{t2.6}
Let there exist $\kappa>0$ such that \eqref{2.13} and \eqref{2.13.5} hold,
where $q\in\kloctnrp$ is nondecreasing in the second argument and 
satisfies \eqref{3.5} for every $b>t_0$. Let, moreover,
\eqref{2.14} and \eqref{2.15.r} be fulfilled. Assume further that 
\begin{gather}
\label{2.20}
p_0(t)\geq p_1(t)\forae t\leq t_0,\\
\label{2.21}
p_1(t)\big(\mu_0(t)-\mu_1(t)\big)\geq 0\forae t\leq t_0.
\end{gather}
Then, for every $c\in[0,\kappa]$ there exists a global solution $u$ to the problem
\eqref{1.3}, \eqref{1.2} satisfying \eqref{2.8} and \eqref{2.9}.
\end{theo}

\begin{rem}
\label{r2.5}
Note that the solution $u$, the existence of which is guaranteed by
Theorem~\ref{t2.6}, has also a finite limit $u(-\infty)$ because $u$
is a bounded nondecreasing function.
\end{rem}

\begin{theo}
\label{t2.7}
Let all the assumptions of Theorem~\ref{t2.5} be fulfilled. Let,
moreover, there exist $g\in\llocrp$ and a continuous nondecreasing
function $h_0:(0,+\infty)\to (0,+\infty)$ such that
\begin{equation}
\label{2.22}
h(t,x,y)\leq g(t)h_0(x+y)\forae t\leq t_0, \quad x,y\in\RR_+,\quad x+y\not=0
\end{equation}
and \eqref{2.11} holds. Then, for every $c\in \left(0,\kappa e^{-M_{\mu}}\right)$
with $M_{\mu}$ given by \eqref{2.19}, there exists a global solution $u$
to the problem \eqref{1.3}, \eqref{1.2} satisfying \eqref{2.12}.
\end{theo}

\begin{rem}
\label{r2.4.5}
Note that the solution $u$, the existence of which is guaranteed by
Theorem~\ref{t2.7}, has also a finite limit $u(-\infty)$ (see
Theorem~\ref{t2.14}).
\end{rem}

\begin{theo}
\label{t2.8}
Let all the assumptions of Theorem~\ref{t2.6} be fulfilled. Let,
moreover, there exist $g\in\llocrp$ and a continuous nondecreasing
function $h_0:(0,+\infty)\to (0,+\infty)$ such that \eqref{2.11} and \eqref{2.22} hold.
Then, for every $c\in(0,\kappa]$ there exists a global solution $u$ to the problem
\eqref{1.3}, \eqref{1.2} satisfying \eqref{2.9} and \eqref{2.12}.
\end{theo}

\begin{rem}
\label{r2.5.5}
Note that the solution $u$, the existence of which is guaranteed by
Theorem~\ref{t2.8}, has also a finite limit $u(-\infty)$ because $u$
is a bounded nondecreasing function.
\end{rem}

\begin{theo}
\label{t2.5.r}
Let there exist $\kappa>0$ such that \eqref{2.13} and \eqref{2.13.5}
hold where $q\in\kloctnrp$ is nondecreasing in the second argument and
satisfies \eqref{3.5} for every $b>t_0$. Let, moreover, \eqref{2.14},
\eqref{2.15.r}, \eqref{2.17}, \eqref{2.18}, and
\begin{equation}
\label{2.16.r}
\int_{\mu_1(t)}^t p_1(s)ds\leq \frac{1}{e}\forae t\in\RR
\end{equation}
be fulfilled. Let, in addition,
\begin{equation}
\label{2.13.r}
h(t,x,y)\geq 0\forae t\geq t_0,\quad x,y\in\RR_+.
\end{equation}
Then, for every $c\in \left(0,\kappa e^{-M_{\mu}}\right)$, where
$M_{\mu}$ is given by \eqref{2.19}, there exists a global solution $u$ to the problem \eqref{1.3}, \eqref{1.2}
satisfying \eqref{2.8} and \eqref{2.8.r}.
\end{theo}

\begin{rem}
\label{r2.100}
Note that, according to Theorem~\ref{t2.1.r} (see also the proof of Theorem~\ref{t2.5.r}), in the case when
$\mu_0(t)=t$ for almost every $t\geq t_0$, resp. $\mu_0(t)=\nu(t)$ for
almost every $t\geq t_0$, the condition \eqref{2.13.r} in
Theorem~\ref{t2.5.r} can be weakened to
$$
p_0(t)x+h(t,x,y)\geq 0\forae t\geq t_0,\quad x,y\in\RR_+,
$$
resp.
$$
p_0(t)y+h(t,x,y)\geq 0\forae t\geq t_0,\quad x,y\in\RR_+.
$$
Moreover, the condition \eqref{2.17} in Theorem~\ref{t2.5.r} can be weakened to \eqref{100}
where $\lambda\in[1,e]$ satisfies \eqref{101}.
Obviously, in that case the number $M_{\mu}$ can also be improved in
an appropriate sense.
\end{rem}

\begin{rem}
\label{r2.4.r}
Note that the solution $u$, the existence of which is guaranteed by
Theorem~\ref{t2.5.r}, has also a finite limit $u(-\infty)$ (see
Theorem~\ref{t2.14}).
\end{rem}

\begin{theo}
\label{t2.6.r}
Let there exist $\kappa>0$ such that \eqref{2.13} holds and let
\begin{equation}
\label{2.13.5.r}
0\leq h(t,x,y)\leq q(t,x+y) \forae t\geq t_0,\quad
x,y\in\RR_+
\end{equation}
where $q\in\kloctnrp$ is nondecreasing in the second argument and 
satisfies \eqref{3.5} for every $b>t_0$. Let, moreover,
\eqref{2.14} and \eqref{2.15.r} be fulfilled. Assume further that 
\begin{gather}
\label{2.20.r}
p_0(t)\geq p_1(t)\forae t\in\RR,\\
\label{2.21.r}
p_1(t)\big(\mu_0(t)-\mu_1(t)\big)\geq 0\forae t\in\RR.
\end{gather}
Then, for every $c\in[0,\kappa]$ there exists a global solution $u$ to the problem
\eqref{1.3}, \eqref{1.2} satisfying \eqref{2.8} and \eqref{2.9.r}.
\end{theo}

\begin{rem}
\label{r2.2.1.r2}
Obviously, if $c>0$ in Theorem~\ref{t2.6.r} then \eqref{2.9.r} implies \eqref{2.8.r}.
\end{rem}

\begin{rem}
\label{r2.5.r}
Note that the solution $u$, the existence of which is guaranteed by
Theorem~\ref{t2.6.r}, has also a finite limit $u(-\infty)$ because $u$
is a bounded nondecreasing function.
\end{rem}

Theorems~\ref{t2.7}--\ref{t2.6.r} together with Remark~\ref{r2.2.1.r2}
imply the following results dealing with the existence of global solutions to
\eqref{1.3}, \eqref{1.2} which are positive on the whole real
line.

\begin{cor}
\label{c1.3}
Let all the assumptions of Theorem~\ref{t2.5.r} be fulfilled. Let,
moreover, there exist $g\in\llocrp$ and a continuous nondecreasing
function $h_0:(0,+\infty)\to (0,+\infty)$ such that \eqref{2.11} and \eqref{2.22} hold.
Then, for every $c\in \left(0,\kappa e^{-M_{\mu}}\right)$
with $M_{\mu}$ given by \eqref{2.19}, there exists a positive global solution $u$
to the problem \eqref{1.3}, \eqref{1.2} satisfying \eqref{2.12.r}.
\end{cor}

\begin{cor}
\label{c1.4}
Let all the assumptions of Theorem~\ref{t2.6.r} be fulfilled. Let,
moreover, there exist $g\in\llocrp$ and a continuous nondecreasing function
$h_0:(0,+\infty)\to (0,+\infty)$ such that \eqref{2.11} and \eqref{2.22} hold.
Then, for every $c\in (0,\kappa]$ there exists a positive global solution $u$ to the
problem \eqref{1.3}, \eqref{1.2} satisfying \eqref{2.9.r} and \eqref{2.12.r}.
\end{cor}

\begin{rem}
\label{r2.98}
Note that, according to Remarks~\ref{r2.4.r} and \ref{r2.5.r}, the
solution $u$, the existence of which is guaranteed by 
Corollary~\ref{c1.3}, resp. Corollary~\ref{c1.4}, has also a finite limit $u(-\infty)$.
\end{rem}

Corollaries~\ref{c1.3} and \ref{c1.4} are direct consequences of
Theorems~\ref{t2.7}--\ref{t2.6.r} and
Remark~\ref{r2.2.1.r2}. Therefore, their proofs are omitted.

\subsection{Properties of solutions}

\begin{theo}
\label{t2.13}
Let there exist $\kappa>0$ such that \eqref{2.13} holds. Let, moreover,
\begin{gather}
\label{2.15}
\mu_0(t)\leq t,\qquad \mu_1(t)\leq t,\qquad \nu(t)\leq t
\forae t\leq t_0,\\
\label{2.40}
\lim_{t\to -\infty}\int_t^{t_0}p_1(s)ds<+\infty,
\end{gather}
and \eqref{2.16} be fulfilled. Then every solution $u$ to \eqref{1.3} on
$(-\infty,t_0]$ satisfying \eqref{2.8} has a finite limit $u(-\infty)$.
\end{theo}

\begin{theo}
\label{t2.14}
Let there exist $\kappa>0$ such that \eqref{2.13} holds.
Let, moreover, \eqref{2.16}--\eqref{2.18} and \eqref{2.15} be fulfilled.
Then every solution $u$ to \eqref{1.3} on $(-\infty,t_0]$ satisfying \eqref{2.8} has a
finite limit $u(-\infty)$.
\end{theo}

\begin{theo}
\label{t2.15}
Let there exist $\kappa>0$, $g\in\llocrp$, and a continuous function
$h_1:(0,\kappa)\times(0,\kappa)\to \RR$ such that
\begin{gather}
\label{2.129}
h_1(x,x)>0 \for x\in (0,\kappa),\\
\label{2.41}
h(t,x,y)\geq g(t)h_1(x,y)\forae t\leq t_0,\quad x,y\in (0,\kappa).
\end{gather}
Let, moreover, \eqref{2.15} and \eqref{2.40} be fulfilled.
Assume, further, that $u$ is a solution to \eqref{1.3} on
$(-\infty,t_0]$ having a limit $u(-\infty)\in[0,\kappa]$. If either
\begin{equation}
\label{2.42}
\lim_{t\to -\infty}\int_t^{t_0}p_0(s)ds=+\infty
\end{equation}
or \eqref{2.38} holds
then either \eqref{2.32} is fulfilled or 
\begin{equation}
\label{2.43}
u(-\infty)=\kappa.
\end{equation}
\end{theo}

\begin{theo}
\label{t2.16}
Let there exist $\kappa>0$, $g\in\llocrp$, and a continuous function
$h_1:(0,\kappa)\times(0,\kappa)\to \RR$ such that \eqref{2.129} and \eqref{2.41} hold.
Let, moreover, \eqref{2.16}, \eqref{2.20}, and \eqref{2.15} be
fulfilled. Assume, further, that $u$ is a solution to \eqref{1.3} on
$(-\infty,t_0]$ having a limit $u(-\infty)\in[0,\kappa]$. 
If there exists $\omega\in\Sigma$ such that
\begin{equation}
\label{2.44}
\sup\left\{\int_{\omega(t)}^tp_1(s)ds:t\leq t_0\right\}<+\infty,
\end{equation}
and either
\begin{equation}
\label{2.45}
\limsup_{t\to -\infty}\int_{\omega(t)}^t \big[p_0(s)-p_1(s)\big]ds>0
\end{equation}
or 
\begin{equation}
\label{2.46}
\limsup_{t\to -\infty}\int_{\omega(t)}^t g(s)ds>0
\end{equation}
then either \eqref{2.32} or \eqref{2.43} holds.
\end{theo}

\begin{cor}
\label{c2.3}
Let there exist $\kappa>0$, $g\in\llocrp$, and a continuous function
$h_1:(0,\kappa)\times(0,\kappa)\to \RR_+$ such that \eqref{2.129} and \eqref{2.41} hold. Let, moreover,
\eqref{2.16}, \eqref{2.15}, and \eqref{2.40} be fulfilled.
Assume, further, that either \eqref{2.38} or \eqref{2.42} is satisfied.
Then every solution $u$ to \eqref{1.3} on $(-\infty,t_0]$ satisfying
\eqref{2.8} has a finite limit $u(-\infty)$ and either \eqref{2.32} or \eqref{2.43} holds. 
\end{cor}

\begin{cor}
\label{c2.4}
Let all the assumptions of Corollary~\ref{c2.3} be fulfilled. If, in
addition, \eqref{2.20} holds, then every solution $u$ to \eqref{1.3}
on $(-\infty,t_0]$ satisfying \eqref{2.8} has a finite limit
$u(-\infty)$ and either \eqref{2.32} holds or  
\begin{equation}
\label{2.47}
u(t)=\kappa\for t\leq t_0.
\end{equation} 
\end{cor}

\begin{cor}
\label{c2.5}
Let there exist $\kappa>0$, $g\in\llocrp$, and a continuous function
$h_1:(0,\kappa)\times(0,\kappa)\to \RR_+$ such that \eqref{2.129} and \eqref{2.41} holds. Let, moreover,
\eqref{2.16}--\eqref{2.18} and \eqref{2.15} be fulfilled.
Assume, further, that 
$$
\sup\left\{\int_{t-1}^t p_1(s)ds: t\leq t_0\right\}<+\infty,
$$
and either 
$$
\limsup_{t\to -\infty}\int_{t-1}^t\big[p_0(s)-p_1(s)\big]ds>0
$$
or 
$$
\limsup_{t\to -\infty}\int_{t-1}^t g(s)ds>0.
$$
Then every solution $u$ to \eqref{1.3} on $(-\infty,t_0]$ satisfying
\eqref{2.8} has a finite limit $u(-\infty)$ and either \eqref{2.32} or
\eqref{2.47} holds. 
\end{cor}

\begin{rem}
\label{r2.7}
Note that \eqref{2.47} can be fulfilled only if
\begin{gather*}
p_0(t)=p_1(t)\forae t\leq t_0,\qquad h(t,\kappa,\kappa)=0\forae t\leq t_0,\\
\int_{\mu_1(t)}^t p_1(s)ds=0\forae t\leq t_0
\end{gather*}
provided all the assumptions of Corollary~\ref{c2.4} or Corollary~\ref{c2.5} are fulfilled.
\end{rem}

\section{Auxiliary Propositions}
\label{sec4}

\subsection{Preliminaries}

First we introduce some already known results which will be used later.

\begin{defn}
\label{d2.2} 
Let $a,b\in\RR$, $a<b$. A linear continuous operator 
$\ell:\cloc\to\lloc$ is said to belong to the set $\saba$,
resp. $\sabb$, if every function $u\in\acabr$ satisfying
\begin{equation}
\label{2.48}
u'(t)\geq \ell(u)(t)\forae\tin,\qquad
u(a)\geq 0,
\end{equation}
resp.
$$
u'(t)\leq \ell(u)(t)\forae\tin,\qquad
u(b)\geq 0,
$$
admits the inequality
\begin{equation}
\label{2.49}
u(t)\geq 0\for\tin.
\end{equation}
\end{defn}

\begin{defn}
\label{d2.3}
Let $a,b\in\RR$, $a<b$. A linear continuous operator 
$\ell:\cloc\to\lloc$ is said to belong to the set $\sabac$ if every
function $u\in\acabr$ satisfying \eqref{2.48} admits the inequalities
\eqref{2.49} and 
$$
u'(t)\geq 0\forae\tin.
$$
\end{defn}

\begin{prop}[see {\cite[Corollary~1.1]{sab}}]
\label{p2.2}
Let $a\in\RR$, $a<t_0$, $\ell_0\in\vtn$. Then $\ell_0\in\sata$.
\end{prop}

\begin{prop}[see {\cite[Theorem~1.2]{sab}}]
\label{p2.3}
Let $a\in\RR$, $a<t_0$, $\ell_1\in\vtn$, and let there exist a
function $\gamma\in\acaton$ such that
\begin{equation}
\label{2.50}
\gamma'(t)\leq -\ell_1(\gamma)(t)\forae t\in[a,t_0].
\end{equation} 
Then $-\ell_1\in\sata$.
\end{prop}

\begin{prop}[see {\cite[Theorem~1.4]{sab}}]
\label{p2.4}
Let $a\in\RR$, $a<t_0$, $\ell_0\in\sata$, $-\ell_1\in\sata$. Then 
\begin{equation}
\label{2.51}
\ell_0-\ell_1\in\sata.
\end{equation}
\end{prop}

\begin{prop}[see {\cite[Theorem~1.5]{sab}}]
\label{p2.5}
Let $a\in\RR$, $a<t_0$. Then $-\ell_1\in\satt$ if and only if there
exists $\gamma\in\acaton$ satisfying \eqref{2.50}.
\end{prop}

\begin{rem}
\label{r2.8}
Note that, according to Proposition~\ref{p2.5}, we have 
$-\ell_1\in\satautau$ for every $\tau\in (a,t_0)$ provided
$-\ell_1\in\satt\cap\vtn$. 
\end{rem}

\begin{prop}[see {\cite[Theorem~6.2]{cmj}}]
\label{p2.6}
Let $a\in\RR$, $a<t_0$, $\ell_0-\ell_1\in\ptnplus$,
$\ell_0\in\sata$. Then the problem 
\begin{equation}
\label{2.52}
u'(t)=\ell_0(u)(t)-\ell_1(u)(t), \qquad u(t_0)=0
\end{equation}
has on $[a,t_0]$ only the trivial solution.
\end{prop}

\begin{prop}[see {\cite[Theorem~4.1]{cmj}}]
\label{p2.7}
Let $a\in\RR$, $a<t_0$, $\ell_0-\ell_1\in\ptnplus$, $\ell_0\in\sata$. Then 
\begin{equation}
\label{2.53}
\ell_0-\ell_1\in\satac.
\end{equation}
\end{prop}

\begin{lemma}
\label{l2.1}
Let $\ell:\cloc\to\lloc$ be a linear positive\footnote{It transforms
  non-negative functions into the set of non-negative functions.}
continuous operator, $\ell\in\vtn$, $\alpha\in\cloc$ be a non-negative
function, and let $\beta\in\cloc$ be a nondecreasing function. Then
\begin{equation}
\label{2.54}
\ell(\alpha\beta)(t)\leq\ell(\alpha)(t)\beta(t)\forae t\leq t_0.
\end{equation}
\end{lemma}

\begin{proof}
Let $A$ be a set of those points $t\in (-\infty,t_0]$ where the
derivatives 
$$
\frac{d}{dt}\int_t^{t_0}\ell(\alpha\beta)(s)ds\qquad\mbox{and}\qquad
\frac{d}{dt}\int_t^{t_0}\ell(\alpha)(s)ds
$$
exist and are equal to $\ell(\alpha\beta)(t)$ and $\ell(\alpha)(t)$,
respectively. Let $t\in A$ be arbitrary but fixed. According to the
inclusion $\ell\in\vtn$ we have
\begin{equation}
\label{2.55}
\ell(\alpha\beta)(s)\leq \ell(\alpha)(s)\beta(t)\forae s\leq t.
\end{equation}
Consequently, from \eqref{2.55} it follows that
\begin{equation}
\label{2.56}
\frac{1}{h}\int_{t-h}^t\ell(\alpha\beta)(s)ds\leq 
\frac{\beta(t)}{h}\int_{t-h}^t\ell(\alpha)(s)ds\for h>0.
\end{equation}
Passing to the limit as $h$ tends to zero in \eqref{2.56},
we get
$$
\ell(\alpha\beta)(t)\leq \ell(\alpha)(t)\beta(t).
$$
Since $t\in A$ was arbitrary, from the latter inequality it follows
that \eqref{2.54} holds.
\end{proof}

\subsection{Lemmas on a finite interval}

\begin{lemma}
\label{l2.2} Let $a\in\RR$, $a<t_0$, $\ell_0,\ell_1\in\vtn$,
and let 
\begin{equation}
\label{2.57}
-\ell_1\in\satt.
\end{equation}
Then the problem \eqref{2.52} has on $[a,t_0]$ only the trivial solution.
\end{lemma}

\begin{proof}
First note that according to Propositions~\ref{p2.2}--\ref{p2.5}, in
view of \eqref{2.57}, we have \eqref{2.51}.
Let $u$ be a solution to the problem \eqref{2.52} on
$[a,t_0]$. Obviously, without loss of generality we can assume that 
\begin{equation}
\label{2.58}
u(a)\geq 0.
\end{equation}
According to \eqref{2.51}, in view of \eqref{2.52} and \eqref{2.58},
we have
\begin{equation}
\label{2.59}
u(t)\geq 0\for t\in[a,t_0].
\end{equation}
Therefore, from \eqref{2.52} it follows that 
\begin{equation}
\label{2.60}
u'(t)\geq -\ell_1(u)(t)\forae t\in [a,t_0], \qquad u(t_0)=0.
\end{equation}
However, according to \eqref{2.57} and \eqref{2.60}, we have
\begin{equation}
\label{2.61}
u(t)\leq 0\for t\in [a,t_0].
\end{equation}
Now \eqref{2.59} and \eqref{2.61} results in $u\equiv 0$.
\end{proof}

\begin{lemma}
\label{l2.3}
Let $a\in\RR$, $a<t_0$, $\ell_1\in\vtn$, and let \eqref{2.29} and
\eqref{2.57} be fulfilled. Let, moreover, $u\in\acatrp$ satisfy
\begin{equation}
\label{2.62}
u'(t)\geq \ell_0(u)(t)-\ell_1(u)(t)\forae t\in [a,t_0].
\end{equation}
Then
\begin{equation}
\label{2.63}
u(a)=\min\big\{u(t):t\in [a,t_0]\big\},
\end{equation}
and, in addition, if there exists $\tau\in (a,t_0]$ such that
$u(\tau)=u(a)$, then
\begin{equation}
\label{2.64}
u(t)=u(a) \for t\in[a,\tau].
\end{equation}
\end{lemma}

\begin{proof}
To prove lemma it is sufficient to show that whenever there exists
$\tau\in(a,t_0]$ such that
\begin{equation}
\label{2.65}
u(\tau)=\min\big\{u(t):t\in [a,t_0]\big\}
\end{equation} 
then $u$ satisfies \eqref{2.64}, and so \eqref{2.63} holds
necessarily. Therefore, let $\tau\in (a,t_0]$ be arbitrary but fixed, such
that \eqref{2.65} holds. Put 
\begin{equation}
\label{2.66}
z(t)=u(t)-u(\tau)\for t\in [a,\tau].
\end{equation}
Then, in view of \eqref{2.29}, \eqref{2.62}, \eqref{2.65}, and
\eqref{2.66}, we have
\begin{gather}
\label{2.67}
z(t)\geq 0\for t\in[a,\tau],\\
\label{2.68}
z'(t)\geq \ell_0(1)(t)u(\tau)-\ell_1(u)(t)\geq -\ell_1(z)(t)\forae
t\in[a,\tau],\\
\label{2.69}
z(\tau)=0.
\end{gather}
Moreover, according to Remark~\ref{r2.8}, the inclusion \eqref{2.57}
implies 
\begin{equation}
\label{2.70}
-\ell_1\in\satautau.
\end{equation}
Therefore, from \eqref{2.68} and \eqref{2.69} we
get
\begin{equation}
\label{2.71}
z(t)\leq 0\for t\in[a,\tau].
\end{equation}
Now \eqref{2.66}, \eqref{2.67}, and \eqref{2.71} implies \eqref{2.64}.
\end{proof}

\begin{lemma}
\label{l2.4}
Let $a,\tau\in\RR$, $a<\tau$, and let there exist $\gamma\in\acatauon$
satisfying 
\begin{equation}
\label{2.72}
\gamma'(t)\leq -\ell_1(\gamma)(t)\forae t\in [a,\tau].
\end{equation}
Let, moreover, $u\in\acataur$ be such that
\begin{gather}
\label{2.73}
\max\big\{u(t):t\in [a,\tau]\big\}\geq 0,\\
\label{2.74}
u'(t)\geq -\ell_1(u)(t)\forae t\in [a,\tau].
\end{gather}
Then
\begin{equation}
\label{2.75}
\max\left\{\frac{u(t)}{\gamma(t)}:t\in [a,\tau]\right\}=\frac{u(\tau)}{\gamma(\tau)}.
\end{equation}
\end{lemma}

\begin{proof}
Put
\begin{equation}
\label{2.76}
\lambda=\max\left\{\frac{u(t)}{\gamma(t)}:t\in [a,\tau]\right\}.
\end{equation}
Then, according to \eqref{2.72}--\eqref{2.74} and
\eqref{2.76} we have $\lambda\geq 0$,
\begin{gather}
\label{2.77}
\lambda\gamma(t)-u(t)\geq 0\for t\in [a,\tau],\\
\label{2.78}
\lambda\gamma'(t)-u'(t)\leq -\ell_1(\lambda\gamma-u)(t)\forae t\in [a,\tau],
\end{gather}
and there exists $\tau_0\in[a,\tau]$ such that
\begin{equation}
\label{2.79}
\lambda\gamma(\tau_0)-u(\tau_0)=0.
\end{equation}
However, from \eqref{2.77} and \eqref{2.78} it follows that
$\lambda\gamma-u$ is a nonincreasing function, which together with
\eqref{2.76}, \eqref{2.77}, and \eqref{2.79} results in \eqref{2.75}.
\end{proof}

Now, from Lemma~\ref{l2.4} we get the following

\begin{lemma}
\label{l2.5}
Let $a\in\RR$, $a<t_0$, $\ell_1\in\vtn$, and let there exist
$\gamma\in\acaton$ satisfying \eqref{2.50}. Let, moreover,
$u\in\acatrp$ be such that
\begin{equation}
\label{2.80}
u'(t)\geq -\ell_1(u)(t)\forae t\in [a,t_0].
\end{equation}
Then
\begin{equation}
\label{2.81}
\ell_1(u)(t)\leq \frac{\ell_1(\gamma)(t)}{\gamma(t)}u(t)\forae t\in [a,t_0].
\end{equation}
\end{lemma}

\begin{proof}
Obviously, since $\ell_1\in\vtn$, from \eqref{2.50} and \eqref{2.80}
it follows that the assumptions of Lemma~\ref{l2.4} are fulfilled for arbitrary
$\tau\in (a,t_0]$. Therefore, according to Lemma~\ref{l2.4} we
have 
$$
\max\left\{\frac{u(t)}{\gamma(t)}:t\in[a,\tau]\right\}=\frac{u(\tau)}{\gamma(\tau)}\for
\tau\in [a,t_0].
$$
However, the latter means that the function $u/\gamma$ is
nondecreasing. Therefore, according to Lemma~\ref{l2.1} with
$\ell=\ell_1$, $\alpha=\vartheta(\gamma)$,
$\beta=\vartheta(u/\gamma)$, and $\vartheta$ given by \eqref{1.4}, 
we obtain \eqref{2.81}.
\end{proof}

\begin{lemma}
\label{l2.6}
Let $a\in\RR$, $a<t_0$, $p\in\latrp$, $\sigma\in\Sigma$, \eqref{2.1}
hold, and let 
\begin{equation}
\label{2.82}
\ell_0(1)(t)\geq p(t)\forae t\in [a,t_0].
\end{equation}
Let, moreover, $u\in\acatrp$ satisfy
\begin{equation}
\label{2.83}
u'(t)\geq \ell_0(u)(t)-p(t)u(t)\forae t\in [a,t_0],
\end{equation}
and let there exist an interval $[\tau_0,\tau_1]\subset (a,t_0]$
such that 
\begin{equation}
\label{2.84}
u(t)>u(\tau_1)\for t\in [\tau_0,\tau_1).
\end{equation}
Then
\begin{equation}
\label{2.85}
\sigma(\tau_1)<\tau_0.
\end{equation}
\end{lemma}

\begin{proof}
Assume on the contrary that 
\begin{equation}
\label{2.86}
\sigma(\tau_1)\geq \tau_0.
\end{equation}
According to the continuity of $u$, in view of \eqref{2.84}, there
exists $\delta\in (0,\tau_0-a)$ such that
\begin{equation}
\label{2.87}
u(t)>u(\tau_1)\for t\in [\tau_0-\delta,\tau_1).
\end{equation}
Furthermore, the continuity of $\sigma$, in view of \eqref{2.86},
yields the existence of $\varepsilon>0$ such that
\begin{equation}
\label{2.88}
\sigma_*\geq \tau_0-\delta
\end{equation}
where
\begin{equation}
\label{2.89}
\sigma_*=\min\big\{\sigma(t):t\in[\tau_1-\varepsilon,\tau_1]\big\}.
\end{equation}
Note that in view of \eqref{2.89} we have
\begin{equation}
\label{2.90}
\sigma_*\leq\sigma(\tau_1-\varepsilon)\leq \tau_1-\varepsilon.
\end{equation}
On the other hand, from \eqref{2.83} we get
\begin{equation}
\label{2.91}
\left(u(t)\exp\left(-\int_t^{\tau_1}p(s)ds\right)\right)'\geq
\ell_0(u)(t)\exp\left(-\int_t^{\tau_1}p(s)ds\right)\mforae t\in [a,t_0].
\end{equation}
The integration of \eqref{2.91} from $\tau_1-\varepsilon$ to $\tau_1$
results in
$$
u(\tau_1)\geq
u(\tau_1-\varepsilon)\exp\left(-\int_{\tau_1-\varepsilon}^{\tau_1}p(s)ds\right)+
\int_{\tau_1-\varepsilon}^{\tau_1}\ell_0(u)(s)\exp\left(-\int_s^{\tau_1}p(\xi)d\xi\right)ds,
$$
whence, on account of \eqref{2.1}, \eqref{2.82}, and
\eqref{2.87}--\eqref{2.89}, we obtain
\begin{equation}
\label{2.92}
u(\tau_1)\geq
u(\tau_1-\varepsilon)\exp\left(-\int_{\tau_1-\varepsilon}^{\tau_1}p(s)ds\right)+
u(\tau_1)\left(1-\exp\left(-\int_{\tau_1-\varepsilon}^{\tau_1}p(s)ds\right)\right).
\end{equation}
However, \eqref{2.92} implies 
$$
u(\tau_1)\geq u(\tau_1-\varepsilon)
$$
which, on account of \eqref{2.88} and \eqref{2.90} contradicts \eqref{2.87}.
\end{proof}

\begin{lemma}
\label{l2.7}
Let $a\in\RR$, $a<t_0$, $p\in\lloctrp$, $\sigma\in\Sigma$, and let
\eqref{2.1} and \eqref{2.82} hold. Let, moreover, $u\in\acatrp$
satisfy \eqref{2.83}. Then, for every $\tau\in (a,t_0)$, the estimate
\begin{equation}
\label{2.93}
u(t)\geq u(\tau)e^{-M_{\sigma}(a,t_0)}\for t\in[\tau,t_0]
\end{equation}
holds, where
\begin{equation}
\label{2.94}
M_{\sigma}(a,t_0)=\max\left\{\int_{\sigma(t)}^t p(s)ds: t\in [a,t_0]\right\}.
\end{equation}
\end{lemma}

\begin{proof}
Assume on the contrary that \eqref{2.93} is not valid, i.e.,
there exist $\tau_0\in(a,t_0)$ and $\tau_1\in (\tau_0,t_0]$ such that
\begin{equation}
\label{2.95}
u(\tau_1)<u(\tau_0)e^{-M_{\sigma}(a,t_0)}.
\end{equation}
Obviously, without loss of generality we can assume that \eqref{2.84}
is fulfilled. Therefore, according to Lemma~\ref{l2.6} we have
\eqref{2.85}.

On the other hand, from \eqref{2.83} we get
$$
u'(t)\geq -p(t)u(t)\forae t\in [a,t_0],
$$
whence we obtain
$$
u(\tau_1)\geq u(\tau_0)\exp\left(-\int_{\tau_0}^{\tau_1}p(s)ds\right).
$$
However, on account of \eqref{2.85}, the latter inequality yields
$$
u(\tau_1)\geq u(\tau_0)\exp\left(-\int_{\sigma(\tau_1)}^{\tau_1}p(s)ds\right),
$$
which, together with \eqref{2.94} contradicts \eqref{2.95}.
\end{proof}

\begin{lemma}
\label{l2.7.1}
Let $a\in\RR$, $a<t_0$, and let $\gamma\in\acloc$ satisfy
\eqref{2.4}. Then \eqref{2.57} holds.
\end{lemma}

\begin{proof}
Let $\gamma_a:[a,t_0]\to (0,+\infty)$ be a restriction of $\gamma$ to
the interval $[a,t_0]$. From \eqref{2.4} it follows that $\gamma$ is a
nonincreasing function, and so
$$
\vartheta(\gamma_a)(t)\leq \vartheta(\gamma)(t)\for t\in\RR
$$
where $\vartheta$ is given by \eqref{1.4}. Therefore,
\begin{equation}
\label{2.98}
\ell_1(\gamma_a)(t)\leq \ell_1(\gamma)(t)\forae t\in[a,t_0].
\end{equation}
Consequently, from \eqref{2.4}, in view of \eqref{2.98}, it follows
that
\begin{equation}
\label{2.99}
\gamma'_a(t)\leq -\ell_1(\gamma_a)(t)\forae t\in [a,t_0].
\end{equation}
Thus, according to Proposition~\ref{p2.5} the inclusion \eqref{2.57}
holds. 
\end{proof}

\begin{lemma}
\label{l2.8}
Let $\ell_1\in \vtn$, $\sigma\in\Sigma$, \eqref{2.1} hold, and let there
exist $\gamma\in\acloc$ such that \eqref{2.4}--\eqref{2.6} hold. Let,
moreover, $a\in\RR$, $a<t_0$, $\kappa>0$, and let $u\in\acatrp$ satisfy \eqref{2.62} and
\begin{equation}
\label{2.96}
u(t_0)\leq \kappa e^{-M_{\sigma}}
\end{equation}
where $M_{\sigma}$ is given by \eqref{2.7}. Then
\begin{equation}
\label{2.97}
u(t)\leq \kappa\for t\in [a,t_0].
\end{equation}
\end{lemma}

\begin{proof}
According to Lemma~\ref{l2.7.1} we have \eqref{2.57}.
Furthermore, from \eqref{2.4} it follows that $\gamma$ is a nonincreasing
function. Thus, according to Lemma~\ref{l2.1} with $\ell=\ell_1$,
$\alpha=\vartheta(\gamma)$, $\beta=\vartheta(1/\gamma)$, $\vartheta$
given by \eqref{1.4}, from \eqref{2.5} it follows that \eqref{2.29} is
fulfilled. Therefore, according to Lemma~\ref{l2.3} we have 
\begin{equation}
\label{2.100}
u(a)\leq u(t_0).
\end{equation}
Assume that \eqref{2.97} does not hold. Then, on account of
\eqref{2.96} and \eqref{2.100}, there exists $\tau\in(a,t_0)$ such
that
\begin{equation}
\label{2.101}
u(\tau)>\kappa.
\end{equation}
Let $\gamma_a:[a,t_0]\to (0,+\infty)$ be a restriction of $\gamma$ to
the interval $[a,t_0]$. From the proof of Lemma~\ref{l2.7.1} it
follows that \eqref{2.98} and \eqref{2.99} hold.
Therefore, according to Lemma~\ref{l2.5}, from \eqref{2.62} we obtain 
\begin{equation}
\label{2.102}
u'(t)\geq \ell_0(u)(t)-\frac{\ell_1(\gamma)(t)}{\gamma(t)}u(t)\forae
t\in [a,t_0].
\end{equation}
Thus, in view of \eqref{2.5}, and \eqref{2.102}, all
the assumptions of Lemma~\ref{l2.7} with $p=\ell_1(\gamma)/\gamma$ are fulfilled. Therefore,
\eqref{2.93} holds with 
\begin{equation}
\label{2.103}
M_{\sigma}(a,t_0)=\max\left\{\int_{\sigma(t)}^t\frac{\ell_1(\gamma)(s)}{\gamma(s)}ds:t\in
[a,t_0]\right\}.
\end{equation}
However, \eqref{2.6}, \eqref{2.7}, and \eqref{2.103} imply
$M_{\sigma}(a,t_0)\leq M_{\sigma}<+\infty$, and so from \eqref{2.93}
it follows that
\begin{equation}
\label{2.104}
u(t_0)\geq u(\tau)e^{-M_{\sigma}}.
\end{equation}
Now \eqref{2.101} and \eqref{2.104} contradicts \eqref{2.96}.
\end{proof}

\subsection{A priori estimates}

The following lemma can be found in \cite{vb}. We formulate it in a
form suitable for us.

\begin{lemma}
\label{l2.9}
Let $a,b\in\RR$, $a<b$, $t_0\in[a,b]$, and let the problem \eqref{2.52} have on
$[a,b]$ only the trivial solution. Let, moreover, there exist $\rho>0$
such that every function $u\in\acabr$ satisfying 
\begin{gather}
\label{2.105}
u'(t)=\ell_0(u)(t)-\ell_1(u)(t)+\lambda f(u)(t)\forae \tin,\\
\label{2.106}
u(t_0)=\lambda c
\end{gather}
for some $\lambda\in(0,1)$, admits the estimate
\begin{equation}
\label{2.107}
\|u\|_{a,b}\leq \rho.
\end{equation}
Then the problem \eqref{1.1}, \eqref{1.2} has at least one solution on $[a,b]$. 
\end{lemma}

\begin{defn}
\label{d3.1}
Let $a,b\in\RR$, $a<b$. We say that a pair of
operators $(\ell_0,\ell_1)$ belongs to the set $\aab$ if there
exists $\rho_0>0$ such that, for any $q^*\in\labrp$
and $c^*\in\RR_+$, every function $u\in\acabr$ satisfying the
inequalities 
\begin{gather}
\label{3.1}
\big[u'(t)-\ell_0(u)(t)+\ell_1(u)(t)\big]\sgn u(t)\leq
q^*(t)\forae\tin,\\
\label{3.2}
0\leq u(a)\leq c^*
\end{gather}
admits the estimate
\begin{equation}
\label{3.3}
\|u\|_{a,b}\leq \rho_0\left(c^*+\int_a^b q^*(s)ds\right).
\end{equation}
\end{defn}

\begin{lemma}
\label{l2.10}
Let $a,b\in\RR$, $a<b$, and let $\ell_0,\ell_1\in\vb$. Then $(\ell_0,\ell_1)\in\aab$.
\end{lemma}

\begin{proof}
Let $q^*\in\labrp$, $c^*\in\RR_+$, and let $u\in\acabr$ satisfy \eqref{3.1} and \eqref{3.2}. Put 
\begin{equation}
\label{2.111}
w(t)=\max\big\{|u(s)|:s\in[a,t]\big\}\for \tin.
\end{equation}
Then, in view of \eqref{3.1} and \eqref{2.111} we have that $w\in\acabrp$,
\begin{gather}
\label{2.112}
w'(t)\geq 0\forae \tin,\\
\label{2.113}
w(t)\geq |u(t)|\for \tin,
\end{gather}
and
\begin{equation}
\label{2.114}
w(a)\leq c^*.
\end{equation}
Put
$$
A=\big\{t\in [a,b]:w(t)=|u(t)|\big\}.
$$
Then
\begin{equation}
\label{2.115}
w'(t)=
\begin{cases}
u'(t)\sgn u(t)&\casforae t\in A,\\
0&\casforae \tin\setminus A.
\end{cases}
\end{equation}
Furthermore, in view of \eqref{2.113}, we have
\begin{equation}
\label{2.116}
(-1)^i\ell_i(u)(t)\sgn u(t)\leq \ell_i(w)(t)\forae \tin\quad (i=0,1).
\end{equation}
Moreover, on account of \eqref{2.112} and the inclusions
$\ell_0,\ell_1\in\vb$, according to Lemma~\ref{l2.1} (with $t_0=b$,
$\ell=\ell_i$, $\alpha\equiv 1$, $\beta=\vartheta(w)$, $\vartheta$
given by \eqref{1.4}) we find
\begin{equation}
\label{2.117}
\ell_i(w)(t)\leq \ell_i(1)(t)w(t)\forae \tin\quad (i=0,1).
\end{equation}
Thus from \eqref{3.1}, on account of 
\eqref{2.115}--\eqref{2.117}, we get 
\begin{equation}
\label{2.118}
w'(t)\leq \big[\ell_0(1)(t)+\ell_1(1)(t)\big]w(t)+q^*(t)\forae \tin.
\end{equation}
Now, from \eqref{2.118} we obtain 
\begin{multline}
\label{2.119}
w(b)\leq
\exp\left(\int_a^b\big[\ell_0(1)(s)+\ell_1(1)(s)\big]ds\right)\\
\times\left(w(a)+\int_a^b q^*(s)\exp\left(-\int_a^s
    \big[\ell_0(1)(\xi)+\ell_1(1)(\xi)\big]d\xi\right)ds\right).
\end{multline}
However, from \eqref{2.119}, in view of \eqref{2.112}--\eqref{2.114},
it follows that \eqref{3.3} holds with 
$$
\rho_0=\exp\left(\int_a^b \big[\ell_0(1)(s)+\ell_1(1)(s)\big]ds\right).
$$
\end{proof}

\begin{lemma}
\label{l2.11} 
Let $a\in\RR$, $a<t_0$, $\ell_0,\ell_1,f\in \vtn$, and let \eqref{2.57}
be fulfilled. Let, moreover,
\begin{gather}
\label{2.120}
f(v)(t)\geq 0\forae t\in [a,t_0],\quad v\in\catr,\\
\label{2.121}
f(v)(t)=0 \forae t\in [a,t_0],\quad -v\in\catrp.
\end{gather}
Then, every solution $u$ to \eqref{1.1} on $[a,t_0]$ satisfying
\begin{equation}
\label{2.122}
u(t_0)\geq 0,
\end{equation}
admits also the inequality \eqref{2.59}.
\end{lemma}

\begin{proof}
First note that, according to Propositions~\ref{p2.2}--\ref{p2.5}, in
view of \eqref{2.57} we have \eqref{2.51}. Let $u$ be a solution to
\eqref{1.1} on $[a,t_0]$ satisfying \eqref{2.122}. It is sufficient to
show that \eqref{2.58} holds, because then the assertion follows from
\eqref{2.51}, \eqref{2.58}, and \eqref{2.120}. Therefore, assume on
the contrary that  
\begin{equation}
\label{2.123}
u(a)<0.
\end{equation}
Then, in view of \eqref{2.122}, there exists $\tau\in(a,t_0]$ such that
\begin{equation}
\label{2.124}
u(t)<0\for t\in[a,\tau),\qquad u(\tau)=0.
\end{equation} 
Now, from \eqref{1.1}, on account of \eqref{2.121}, \eqref{2.124}, and
the inclusion $f\in \vtn$, we get
\begin{equation}
\label{2.125}
u'(t)=\ell_0(u)(t)-\ell_1(u)(t)\forae t\in[a,\tau].
\end{equation} 
However, \eqref{2.125}, in view of \eqref{2.124} and the inclusion
$\ell_0\in \vtn$, results in 
\begin{equation}
\label{2.126}
u'(t)\leq -\ell_1(u)(t)\forae t\in[a,\tau],\qquad u(\tau)=0.
\end{equation} 
According to Remark~\ref{r2.8}, the inclusion \eqref{2.57} yields \eqref{2.70},
which together with \eqref{2.126} implies 
\begin{equation}
\label{2.127}
u(t)\geq 0\for t\in[a,\tau].
\end{equation}
However, the inequality \eqref{2.127} contradicts \eqref{2.123}. 
\end{proof}

\begin{lemma}
\label{l2.12} 
Let $a\in\RR$, $a<t_0$, $\ell_0,\ell_1,f\in \vtn$, and let
$\ell_0-\ell_1\in\ptnplus$. Let, moreover, \eqref{2.120} and
\eqref{2.121} be fulfilled. Then, every solution $u$ to \eqref{1.1} on
$[a,t_0]$ satisfying \eqref{2.122} admits also the inequalities
\eqref{2.59} and
\begin{equation}
\label{2.128}
u'(t)\geq 0\forae t\in[a,t_0].
\end{equation}
\end{lemma}

\begin{proof}
First note that, according to Propositions~\ref{p2.2} and \ref{p2.7},
we have \eqref{2.53}. Let $u$ be a solution to \eqref{1.1} on
$[a,t_0]$ satisfying \eqref{2.122}. It is sufficient to show that 
\eqref{2.58} holds, because then the assertion follows from
\eqref{2.53}, \eqref{2.58}, and \eqref{2.120}. Therefore, assume on 
the contrary that \eqref{2.123} holds. Then, in view of \eqref{2.122}, 
there exists $\tau\in(a,t_0]$ such that \eqref{2.124} is
satisfied. Now, from \eqref{1.1}, on account of \eqref{2.121},
\eqref{2.124}, and the inclusion $f\in \vtn$, we get \eqref{2.125}.
However, the inclusion $\ell_0-\ell_1\in\ptnplus$ yields
$\ell_0-\ell_1\in\ptplus$. Moreover, the inclusion $\ell_0\in\vtn$
implies $\ell_0\in\vtau$, and so, according to
Proposition~\ref{p2.2}, we have $\ell_0\in\sataua$. Consequently,
according to Proposition~\ref{p2.7} (with $t_0=\tau$), we have
$\ell_0-\ell_1\in\satauac$, which together with \eqref{2.123} and
\eqref{2.125}, implies $u(\tau)\leq u(a)<0$. However, the latter
contradicts \eqref{2.124}.  
\end{proof}

\begin{lemma}
\label{l2.13}
Let $a\in\RR$, $a<t_0$, $\ell_0,\ell_1,f\in \vtn$, \eqref{2.57} hold,
and let there exist $q\in\latrp$ such that 
\begin{equation}
\label{2.108}
f(v)(t)\sgn v(t)\leq q(t) \forae t\in [a,t_0],\quad
v\in\catr
\end{equation}
is fulfilled. Let, moreover, \eqref{2.29}, \eqref{2.120}, and
\eqref{2.121} hold. Then, for every $c\in\RR_+$, the problem
\eqref{1.1}, \eqref{1.2} has at least one solution $u$ on $[a,t_0]$
satisfying \eqref{2.59}. 
\end{lemma}

\begin{proof}
Let $c\in\RR_+$ be arbitrary but fixed. 
According to Lemmas~\ref{l2.2}, \ref{l2.9}, and \ref{l2.11}, it is
sufficient to show that there exists $\rho>0$ such that every function
$u\in\acatr$ satisfying \eqref{2.106} and 
\begin{equation}
\label{2.105.atn}
u'(t)=\ell_0(u)(t)-\ell_1(u)(t)+\lambda f(u)(t)\forae t\in [a,t_0],
\end{equation}
for some $\lambda\in (0,1)$, admits the estimate 
\begin{equation}
\label{2.107.atn}
\|u\|_{a,t_0}\leq \rho.
\end{equation}
Let, therefore, $\lambda\in(0,1)$ and let
$u\in\acatr$ satisfy \eqref{2.106} and \eqref{2.105.atn}. Then, in view of
\eqref{2.120} we have \eqref{2.62} and, according to Lemma~\ref{l2.11}
we have \eqref{2.59}. Thus, on account of \eqref{2.29}, \eqref{2.57},
\eqref{2.59}, and \eqref{2.62}, all the assumptions of
Lemma~\ref{l2.3} are fulfilled. Therefore, \eqref{2.100}
holds. Finally, according to Lemma~\ref{l2.10}, in view of
\eqref{2.59}, \eqref{2.100}, \eqref{2.106}, \eqref{2.108}, and
\eqref{2.105.atn}, there exists $\rho_0>0$ such that
$$
\|u\|_{a,t_0}\leq \rho_0\left(c+\int_a^{t_0} q(s)ds\right)
$$
holds. Consequently, \eqref{2.107.atn} is fulfilled with 
$\rho=\rho_0\left(c+\int_a^{t_0}q(s)ds\right)$.
\end{proof}

\begin{lemma}
\label{l2.14}
Let $a\in\RR$, $a<t_0$, $\ell_0,\ell_1,f\in \vtn$,
$\ell_0-\ell_1\in\ptnplus$. Let, moreover, \eqref{2.120} and
\eqref{2.121} be satisfied. Then, for every $c\in\RR_+$, the problem
\eqref{1.1}, \eqref{1.2} has at least one solution $u$ on $[a,t_0]$
satisfying \eqref{2.59} and \eqref{2.128}.
\end{lemma}

\begin{proof}
Let $c\in\RR_+$ be arbitrary but fixed. 
According to Propositions~\ref{p2.2} and \ref{p2.6}, and
Lemmas~\ref{l2.9} and \ref{l2.12}, it is sufficient to show that there
exists $\rho>0$ such that every function $u\in\acatr$ satisfying
\eqref{2.106} and \eqref{2.105.atn} for some $\lambda\in (0,1)$, admits the
estimate \eqref{2.107.atn}. Let, therefore, $\lambda\in (0,1)$ and let $u\in\acatr$ satisfy
\eqref{2.106} and \eqref{2.105.atn}. Then, according to Lemma~\ref{l2.12}
we have \eqref{2.59} and \eqref{2.128}. Therefore, on account of
\eqref{2.106}, the estimate \eqref{2.107.atn} is fulfilled with $\rho=c$.
\end{proof}

\begin{lemma}
\label{l3.1}
Let $c\in\RR_+$, $b\in\RR$, $b>t_0$, and let
$(\ell_0,\ell_1)\in\atnb$. Let, moreover, 
\begin{equation}
\label{3.4}
f(v)(t)\sgn v(t)\leq q(t,\|v\|_{t_0,b}) \mforae t\in [t_0,b],\quad
v\in\ctbr,\quad 0\leq v(t_0)\leq c,
\end{equation}
where $q\in\ktnbrp$ satisfies \eqref{3.5}.
Then the problem \eqref{1.1}, \eqref{1.2} has at least one solution on $[t_0,b]$.
\end{lemma}

\begin{proof}
First note that due to the inclusion $(\ell_0,\ell_1)\in\atnb$, the
problem \eqref{2.52} has on $[t_0,b]$ only the trivial solution.

Let $\rho_0$ be the number appearing in
Definition~\ref{d3.1}. According to \eqref{3.5} there exists
$\rho>2c\rho_0$ such that
\begin{equation}
\label{3.6}
\frac{1}{x}\int_{t_0}^b q(s,x)ds<\frac{1}{2\rho_0}\for x>\rho.
\end{equation}
Now assume that a function $u\in\actbr$ satisfies \eqref{2.106} and
\begin{equation}
\label{2.105.tnb}
u'(t)=\ell_0(u)(t)-\ell_1(u)(t)+\lambda f(u)(t)\forae t\in [t_0,b]
\end{equation}
for some $\lambda\in (0,1)$. Then, according to \eqref{3.4} we obtain
that 
\begin{gather}
\label{3.1.tnb}
\big[u'(t)-\ell_0(u)(t)+\ell_1(u)(t)\big]\sgn u(t)\leq
q(t,\|u\|_{t_0,b})\forae t\in [t_0,b],\\
\label{3.2.tnb}
0\leq u(t_0)\leq c.
\end{gather}
Hence, by the inclusion $(\ell_0,\ell_1)\in\atnb$ and \eqref{3.6}, we get the estimate
\begin{equation}
\label{2.107.tnb}
\|u\|_{t_0,b}\leq \rho.
\end{equation}

Since $\rho$ depends neither on $u$ nor on $\lambda$, it follows from
Lemma~\ref{l2.9} that the problem \eqref{1.1}, \eqref{1.2} has at
least one solution on $[t_0,b]$.
\end{proof}

\subsection{Existence of solutions defined on half-lines and on $\RR$}

To formulate the following lemma we need to introduce some notation.
Let $\kappa>0$, $I\subseteq\RR$ be a closed interval. Then, for every
continuous function $v:I\to\RR$, we put
\begin{equation}
\label{2.130}
\overline{f}(v)(t)\stackrel{def}{=}f(\psi(\vartheta(v)))(t)\forae t\in\RR,
\end{equation}
where $\vartheta$ is given by \eqref{1.4} and 
\begin{equation}
\label{2.131}
\psi(v)(t)=
\begin{cases}
\kappa&\casif v(t)>\kappa,\\
v(t)&\casif 0\leq u(t)\leq \kappa,\\
0&\casif v(t)<0
\end{cases}
\for t\in\RR.
\end{equation}
Note that $\overline{f}\in \vtn$ provided $f\in \vtn$. Let
$(a_n)_{n=1}^{+\infty}$ be a sequence of real numbers such that 
\begin{equation}
\label{2.132}
a_n<t_0,\qquad
\lim_{n\to +\infty} a_n=-\infty,
\end{equation}
and consider the auxiliary equation
\begin{equation}
\label{2.133}
u'(t)=\ell_0(u)(t)-\ell_1(u)(t)+\overline{f}(u)(t).
\end{equation}

\begin{lemma}
\label{l2.15}
Let $\ell_0,\ell_1\in \vtn$. Let, moreover, there exist $\kappa>0$ and $c\in[0,\kappa]$ such
that, for every $n\in\NN$, the problem \eqref{2.133}, \eqref{1.2}
has a solution $u_n$ on the interval $[a_n,t_0]$ satisfying
\begin{equation}
\label{2.134}
0\leq u_n(t)\leq \kappa\for t\in[a_n,t_0].
\end{equation}
Then \eqref{1.1}, \eqref{1.2} has at least one solution $u$ on
$(-\infty,t_0]$ satisfying \eqref{2.8}. If, in addition,
\begin{equation}
\label{2.200}
u'_n(t)\geq 0\forae t\in[a_n,t_0],\quad n\in\NN
\end{equation}
then \eqref{1.1}, \eqref{1.2} has at least one solution $u$ on
$(-\infty,t_0]$ satisfying \eqref{2.8} and \eqref{2.9}.
\end{lemma}

\begin{proof}
Let $a\in\RR$ be arbitrary but fixed such that $a<t_0$. Then, in view of
\eqref{2.132} there exists $n_0\in\NN$ such that
$[a,t_0]\subseteq [a_n,t_0]$ for $n\geq n_0$. Further, in view of
\eqref{2.130} and \eqref{2.134}, there exists $q\in\latrp$ such that
\begin{equation}
\label{2.135}
|\overline{f}(u_n)(t)|\leq q(t)\forae t\in[a,t_0],\quad n\geq n_0.
\end{equation}
Thus \eqref{2.133}, on account of \eqref{2.134}, and \eqref{2.135}, results in
\begin{equation}
\label{2.136}
|u_n'(t)|\leq \big[\ell_0(1)(t)+\ell_1(1)(t)\big]\kappa+q(t)\forae
t\in[a,t_0],\quad n\geq n_0.
\end{equation}
Therefore, on account of \eqref{2.134} and \eqref{2.136}, the sequence of
solutions $(u_n)_{n=n_0}^{+\infty}$ is uniformly bounded and
equicontinuous on $[a,t_0]$. Since the interval
$[a,t_0]$ was chosen arbitrarily, according to Arzel\`a-Ascoli
theorem, without loss of generality we can assume that there exists
$u\in\cloc$ such that 
\begin{equation}
\label{2.137}
\lim_{n\to
  +\infty}\vartheta(u_n)(t)=u(t)\qquad\mbox{uniformly~on~every~compact~interval}.
\end{equation}
On the other hand, from \eqref{1.2} and \eqref{2.133}, in view of
\eqref{2.130}--\eqref{2.132} and \eqref{2.134} we get 
\begin{equation}
\label{2.138}
u_n(t)=c-\int_t^{t_0}\big[\ell_0(u_n)(s)-\ell_1(u_n)(s)+f(u_n)(s)\big]ds\mfor
t\in[a,t_0],\quad n\geq n_0.
\end{equation}
Thus \eqref{2.137} and \eqref{2.138} yields
$$
u(t)=c-\int_t^{t_0}\big[\ell_0(u)(s)-\ell_1(u)(s)+f(u)(s)\big]ds\for
t\in[a,t_0].
$$
Therefore, because the interval $[a,t_0]$ was chosen
arbitrarily, we have that $u\in\aclocr$ (note that, according to
\eqref{2.137}, $u(t)=c$ for $t\geq t_0$) and the restriction of
$u$ to the interval $(-\infty,t_0]$ is a solution to the problem
\eqref{1.1}, \eqref{1.2} on $(-\infty,t_0]$. Obviously, according to
\eqref{2.134}, $u$ satisfies \eqref{2.8} being a limit of $\vartheta(u_n)$. 

Moreover, if \eqref{2.200} holds then, for every $n\in\NN$,
$$
\vartheta(u_n)(\tau_1)\leq \vartheta(u_n)(\tau_2)\qquad\mbox{whenever~}\,\tau_1\leq \tau_2,
$$
and so, in view of \eqref{2.137}, we
have
$$
u(\tau_1)\leq u(\tau_2)\qquad\mbox{whenever~}\,\tau_1\leq \tau_2.
$$
Therefore, \eqref{2.9} holds.  
\end{proof}

\begin{lemma}
\label{l3.2}
Let $\ell_0,\ell_1,f\in V$, $c\in\RR_+$, and let 
\begin{multline}
\label{3.7}
f(v)(t)\sgn v(t)\leq q(t,\|\vartheta(v)\|) \forae t\geq t_0,\\
v\in\cnultnr,\quad 0\leq v(t_0)\leq c,
\end{multline}
where $q\in\kloctnrp$ satisfies \eqref{3.5} for every $b>t_0$.
Then the problem \eqref{1.1}, \eqref{1.2} has at least one solution on $[t_0,+\infty)$.
\end{lemma}

\begin{proof}
Note that, in view of \eqref{3.7}, we have that \eqref{3.4} holds for
every $b\in\RR$, $b>t_0$. Therefore, according to Lemmas~\ref{l2.10}
and \ref{l3.1}, the problem \eqref{1.1}, \eqref{1.2} has at least one
solution on $[t_0,b]$ for every $b>t_0$.

Let $(b_n)_{n=1}^{+\infty}$ be a sequence of real numbers such that 
\begin{equation}
\label{3.8}
b_n>t_0,\qquad
\lim_{n\to +\infty} b_n=+\infty,
\end{equation}
and let, for every $n\in\NN$, $u_n$ be a solution to \eqref{1.1},
\eqref{1.2} on $[t_0,b_n]$. Let, moreover, $b\in\RR$ be arbitrary but
fixed such that $b>t_0$. Then, in view of \eqref{3.8}, there exists
$n_0\in\NN$ such that $[t_0,b]\subseteq [t_0,b_n]$ for $n\geq
n_0$. Further, let, for every $n\geq n_0$, $\overline{u}_n$ be a
restriction of $u_n$ to the interval $[t_0,b]$. Then, in view of the
inclusion $\ell_0,\ell_1,f\in V$ we have
\begin{gather}
\label{3.9}
\overline{u}'_n(t)=\ell_0(\overline{u}_n)(t)-\ell_1(\overline{u}_n)(t)+f(\overline{u}_n)(t) 
\forae t\in [t_0,b],\quad n\geq n_0,\\
\label{3.10}
\overline{u}_n(t_0)=c\for n\geq n_0.
\end{gather}
According to Lemma~\ref{l2.10} we have $(\ell_0,\ell_1)\in\atnb$. Let $\rho_0$ be the number appearing in
Definition~\ref{d3.1}. According to \eqref{3.5} there exists
$\rho>2c\rho_0$ such that \eqref{3.6} holds. Thus, according to
\eqref{3.4}, \eqref{3.9}, and \eqref{3.10}, for every $n\geq n_0$ we obtain
\begin{gather*}
\big[\overline{u}'_n(t)-\ell_0(\overline{u}_n)(t)+\ell_1(\overline{u}_n)(t)\big]\sgn
\overline{u}_n(t)\leq q(t,\|\overline{u}_n\|_{t_0,b})\forae t\in [t_0,b],\\
0\leq \overline{u}_n(t_0)\leq c.
\end{gather*}
Hence, accroding to $(\ell_0,\ell_1)\in\atnb$ and \eqref{3.6}, we get the estimate
\begin{equation}
\label{3.11}
\|\overline{u}_n\|_{t_0,b}\leq \rho\for n\geq n_0.
\end{equation}
Moreover, using \eqref{3.11} in \eqref{3.9} we get
\begin{equation}
\label{3.12}
|\overline{u}'_n(t)|\leq \big[\ell_0(1)(t)+\ell_1(1)(t)\big]\rho+f^*(t)\forae
t\in[t_0,b],\quad n\geq n_0
\end{equation}
where $f^*\in\ltbrp$ is such that 
$$
|f(v)(t)|\leq f^*(t)\forae t\in[t_0,b],\quad v\in\ctbr,\quad
\|v\|_{t_0,b}\leq \rho.
$$
Consequently, from \eqref{3.11} and \eqref{3.12} it follows that the
sequence of solutions $(u_n)_{n=n_0}^{+\infty}$ is uniformly bounded and
equicontinuous on $[t_0,b]$. Since the interval $[t_0,b]$ was chosen
arbitrarily, according to Arzel\`a-Ascoli theorem, without loss of
generality we can assume that there exists $u\in\cloc$ such that
\eqref{2.137} holds.

On the other hand, from \eqref{1.1} and \eqref{1.2}, we have 
\begin{equation}
\label{3.14}
u_n(t)=c+\int_{t_0}^t\big[\ell_0(u_n)(s)-\ell_1(u_n)(s)+f(u_n)(s)\big]ds\mfor
t\in[t_0,b],\quad n\geq n_0.
\end{equation}
Thus, \eqref{2.137} and \eqref{3.14} yield
$$
u(t)=c+\int_{t_0}^t\big[\ell_0(u)(s)-\ell_1(u)(s)+f(u)(s)\big]ds\for
t\in[t_0,b].
$$
Therefore, because the interval $[t_0,b]$ was chosen
arbitrarily, we have that $u\in\aclocr$ (note that, according to
\eqref{2.137}, $u(t)=c$ for $t\leq t_0$) and the restriction of
$u$ to the interval $[t_0,+\infty)$ is a solution to the problem
\eqref{1.1}, \eqref{1.2} on $[t_0,+\infty)$.
\end{proof}

\begin{lemma}
\label{l3.3}
Let $\ell_0,\ell_1,f\in V$, $\kappa>0$, $c\in[0,\kappa]$, and let
\begin{multline}
\label{3.15}
f(v)(t)\sgn v(t)\leq q(t,\|v\|) \mforae t\geq t_0,\quad
v\in\cnul,\\ 
0\leq v(t_0)\leq c,\quad -c\leq v(t)\leq \kappa\mfor t\leq t_0,
\end{multline}
where $q\in\kloctnrp$ is nondecreasing in the second argument and
satisfies \eqref{3.5} for every $b>t_0$. Let, moreover, there exist a
solution $u_0$ to \eqref{1.1}, \eqref{1.2} on $(-\infty,t_0]$
satisfying
\begin{equation}
\label{3.16}
0\leq u_0(t)\leq \kappa\for t\leq t_0.
\end{equation}
Then the problem \eqref{1.1}, \eqref{1.2} has at least one global solution
$u$ such that
$$
u(t)=u_0(t)\for t\leq t_0.
$$ 
\end{lemma}

\begin{proof}
Consider the auxiliary equation
\begin{equation}
\label{3.17}
u'(t)=\ell_0(u)(t)-\ell_1(u)(t)+\widehat{f}(u)(t)
\end{equation}
where 
\begin{multline}
\label{3.18}
\widehat{f}(v)(t)\stackrel{def}{=}f(\vartheta(u_0)-c+v)(t)
+\ell_0(\vartheta(u_0)-c)(t)-\ell_1(\vartheta(u_0)-c)(t)\\ 
\mforae t\in\RR,\quad v\in\cloc
\end{multline}
where $\vartheta$ is given by \eqref{1.4}. Obviously, $\widehat{f}\in V$. 

Now let $v\in\cnultnr$ be arbitrary but fixed such that
\begin{equation}
\label{3.19}
0\leq v(t_0)\leq c
\end{equation}
and put
\begin{gather}
\label{3.20}
w(t)=\vartheta(u_0)(t)-c+\vartheta(v)(t)\for t\in\RR,\\
\label{3.21}
\widehat{q}(t,x)=q(t,x+\kappa)+\kappa\big[\ell_0(1)(t)+\ell_1(1)(t)\big]\forae t\geq t_0,\quad x\in\RR_+.
\end{gather}
Then, in view of \eqref{3.16}, \eqref{3.19}, and \eqref{3.20} we have
$w\in\cnul$,
\begin{gather}
\label{3.22}
0\leq w(t_0)\leq c,\\
\label{3.23}
-c\leq w(t)\leq \kappa\for t\leq t_0,\\
\label{3.24}
\|\vartheta(u_0)-c\|\leq \kappa.
\end{gather}
Consequently, in view of \eqref{3.15} and \eqref{3.21}--\eqref{3.24}
we have
\begin{equation}
\label{3.25}
f(w)(t)\sgn w(t)\leq q(t,\|w\|)\leq
q(t,\|\vartheta(v)\|+\kappa)\forae t\geq t_0.
\end{equation}
On the other hand,
\begin{equation}
\label{3.26}
\widehat{f}(v)(t)\sgn v(t)=
\big[f(w)(t)+\ell_0(\vartheta(u_0)-c)(t)-\ell_1(\vartheta(u_0)-c)(t)\big]\sgn w(t)\mforae t\geq t_0.
\end{equation}
Therefore, \eqref{3.25} and \eqref{3.26}, in view of \eqref{3.19},
\eqref{3.21}, and \eqref{3.24}, result in 
$$
\widehat{f}(v)(t)\sgn v(t)\leq \widehat{q}(t,\|\vartheta(v)\|) \mforae t\geq t_0,\quad
v\in\cnultnr,\quad 0\leq v(t_0)\leq c.
$$
Moreover, on account of \eqref{3.21},
$$
\frac{1}{x}\int_{t_0}^b\widehat{q}(t,x)dt=\left(1+\frac{\kappa}{x}\right)\frac{1}{x+\kappa}\int_{t_0}^b
q(t,x+\kappa)dt+\frac{\kappa}{x}\int_{t_0}^b \big[\ell_0(1)(t)+\ell_1(1)(t)\big]dt \mfor b>t_0,
$$
and thus, in view of \eqref{3.5}, we have
$$
\lim_{x\to +\infty}\frac{1}{x}\int_{t_0}^b\widehat{q}(t,x)dt=0
$$
for every $b>t_0$. Consequently, all the assumptions of
Lemma~\ref{l3.2} are fulfilled with $f=\widehat{f}$ and
$q=\widehat{q}$. Therefore, the problem \eqref{3.17}, \eqref{1.2} has
at least one solution $u_1$ on $[t_0,+\infty)$.

Now put
$$
u(t)=
\begin{cases}
u_0(t)&\casfor t\leq t_0,\\
u_1(t)&\casfor t>t_0.
\end{cases}
$$
Then, obviously, $u\in\aclocr$ and in view of \eqref{3.17} and
\eqref{3.18}, the function
$u$ is a global solution to \eqref{1.1}, \eqref{1.2}.
\end{proof}

\subsection{Properties of a solution in the neighbourhood of $-\infty$}

\begin{lemma}
\label{l2.16}
Let $\ell_0,f\in \vtn$, and let there exist $g\in\llocrp$ and a
continuous nondecreasing function $h_0:(0,+\infty)\to (0,+\infty)$
satisfying \eqref{2.10} and \eqref{2.11}. Let, moreover, $u$ be a
non-negative solution to \eqref{1.1} on $(-\infty,t_0]$ such that 
\begin{equation}
\label{2.139}
u(t)=0\for t\leq \tau,
\end{equation}
for some $\tau\in (-\infty,t_0)$. Then $u\equiv 0$.
\end{lemma}

\begin{proof}
Suppose on the contrary that $u$ assumes positive values. Put
\begin{equation}
\label{2.140}
w(t)=\sup\big\{u(s):s\leq t\big\}\for t\leq t_0.
\end{equation}
Then
\begin{gather}
\label{2.141}
w\in\aclocrp,\\
\label{2.142}
w'(t)\geq 0\forae t\leq t_0,\\
\label{2.143}
w(t)\geq u(t)\for t\leq t_0,
\end{gather}
and there exists $\tau_0\in[\tau,t_0)$ such that 
\begin{equation}
\label{2.144}
w(t)=0\for t\leq \tau_0,\qquad w(t)>0\for t\in (\tau_0,t_0].
\end{equation}
Let
$$
A=\big\{t\in[\tau_0,t_0]:w(t)=u(t)\big\}.
$$
Then
\begin{equation}
\label{2.145}
w'(t)=
\begin{cases}
u'(t)&\casforae t\in A,\\
0&\casforae t\in[\tau_0,t_0]\setminus A.
\end{cases}
\end{equation}
Furthermore, in view of \eqref{2.143}, we have
\begin{equation}
\label{2.146}
\ell_0(u)(t)\leq \ell_0(w)(t)\forae t\in[\tau_0,t_0],
\end{equation}
and, on account of \eqref{2.142}, the non-negativity of $u$, and the
inclusion $\ell_0\in \vtn$, according to Lemma~\ref{l2.1} (with
$\ell=\ell_0$, $\alpha\equiv 1$, $\beta=\vartheta(w)$, $\vartheta$
given by \eqref{1.4}), we find
\begin{gather}
\label{2.147}
\ell_0(w)(t)\leq \ell_0(1)(t)w(t)\forae t\in[\tau_0,t_0],\\
\label{2.148}
\ell_1(u)(t)\geq 0\forae t\in[\tau_0,t_0].
\end{gather}
Moreover, from \eqref{2.10}, in view of \eqref{2.140}, \eqref{2.144},
the non-negativity of $u$, and the inclusion $f\in \vtn$, for every
fixed $t\in(\tau_0,t_0]$ we have
\begin{equation}
\label{2.149}
f(u)(s)\leq g(s)h_0(\|u\|_{\tau_0,t})=g(s)h_0(w(t))\forae s\in[\tau_0,t].
\end{equation}
Consequently, analogously to the proof of Lemma~\ref{l2.1}, one can
show that \eqref{2.149} implies
\begin{equation}
\label{2.150}
f(u)(t)\leq g(t)h_0(w(t))\forae t\in(\tau_0,t_0].
\end{equation}
Thus, since $w$, $g$, and $h_0$ are non-negative
functions, from \eqref{1.1}, \eqref{2.145}--\eqref{2.148} and \eqref{2.150} we get
\begin{equation}
\label{2.151}
w'(t)\leq \ell_0(1)(t)w(t)+g(t)h_0(w(t))\forae t\in(\tau_0,t_0].
\end{equation}
However, \eqref{2.151} results in 
\begin{multline}
\label{2.152}
z'(t)\leq
g(t)\exp\left(\int_t^{t_0}\ell_0(1)(s)ds\right)\times\\
h_0\left(z(t)\exp\left(-\int_t^{t_0}\ell_0(1)(s)ds\right)\right)\mforae
t\in(\tau_0,t_0] 
\end{multline}
where
\begin{equation}
\label{2.153}
z(t)=w(t)\exp\left(\int_t^{t_0}\ell_0(1)(s)ds\right)
\for t\in(\tau_0,t_0]. 
\end{equation}
Since $h_0$ is a nondecreasing function, on account of \eqref{2.144}
and \eqref{2.153}, from \eqref{2.152} it follows that
\begin{equation}
\label{2.154}
\frac{z'(t)}{h_0(z(t))}\leq g(t)\exp\left(\int_t^{t_0}\ell_0(1)(s)ds\right)\forae t\in(\tau_0,t_0].
\end{equation}
Now, the integration of \eqref{2.154} from $t$ to $t_0$ yields
$$
\int_{z(t)}^{z(t_0)}\frac{ds}{h_0(s)}\leq
\int_t^{t_0}g(s)\exp\left(\int_s^{t_0}\ell_0(1)(\xi)d\xi\right)ds \for t\in(\tau_0,t_0]
$$
whence we obtain
\begin{equation}
\label{2.155}
\lim_{t\to \tau_{0+}}\int_{z(t)}^{z(t_0)}\frac{ds}{h_0(s)}\leq
\int_{\tau_0}^{t_0}g(s)\exp\left(\int_s^{t_0}\ell_0(1)(\xi)d\xi\right)ds<+\infty.
\end{equation}
However, \eqref{2.155} together with \eqref{2.141}, \eqref{2.144}, and
\eqref{2.153} contradicts \eqref{2.11}.
\end{proof}

\begin{lemma}
\label{l2.17}
Let $\tau\in\RR$ and let there exist $\gamma\in\acloctau$ satisfying
\begin{equation}
\label{2.194}
\gamma'(t)\leq -\ell_1(\gamma)(t)\forae t\leq \tau.
\end{equation}
Let, moreover, $u\in\acloctaur$ satisfy 
\begin{gather}
\label{2.156}
0\leq \sup\big\{u(t):t\leq \tau\big\}<+\infty,\\
\label{2.157}
u'(t)\geq -\ell_1(u)(t)\forae t\leq \tau.
\end{gather}
Then
\begin{equation}
\label{2.158}
\sup\left\{\frac{u(t)}{\gamma(t)}:t\leq \tau \right\}=\frac{u(\tau)}{\gamma(\tau)}.
\end{equation}
\end{lemma}

\begin{proof}
First suppose that $u(t)\leq 0$ for $t\leq \tau$. Then, in view of
\eqref{2.157}, $u$ is a nondecreasing function, which together with
\eqref{2.156} implies $u(\tau)=0$. Therefore, in that case \eqref{2.158} holds.
   
Let, therefore, $u$ assumes positive values. Put
\begin{equation}
\label{2.159}
\lambda=\sup\left\{\frac{u(t)}{\gamma(t)}:t\leq \tau \right\}.
\end{equation}
Then, according to \eqref{2.194}--\eqref{2.157} we have $0<\lambda<+\infty$,
\begin{gather}
\label{2.160}
\lambda\gamma(t)-u(t)\geq 0\for t\leq \tau,\\
\label{2.161}
\lambda\gamma'(t)-u'(t)\leq -\ell_1(\lambda\gamma-u)(t)\forae t\leq \tau.
\end{gather}
According to \eqref{2.194}, $\gamma$ is a nonincreasing
function. Therefore, there exists finite or infinite limit
$\gamma(-\infty)$. If
\begin{equation}
\label{2.162}
\gamma(-\infty)=+\infty
\end{equation}
then, in view of \eqref{2.156} and \eqref{2.159}, there exists $\tau_0\in
(-\infty,\tau]$ such that \eqref{2.79} holds.
If
\begin{equation}
\label{2.163}
\gamma(-\infty)<+\infty
\end{equation}
then, on account of \eqref{2.159}, for every $n\in\NN$ there exists
$\tau_n\in(-\infty,\tau]$ such that
$$
\lambda-\frac{1}{n\gamma(-\infty)}\leq\frac{u(\tau_n)}{\gamma(\tau_n)}
$$
whence, because $\gamma$ is nonincreasing, we get
$$
\lambda\gamma(\tau_n)-u(\tau_n)\leq\frac{\gamma(\tau_n)}{n\gamma(-\infty)}\leq\frac{1}{n}.
$$
Thus, in both cases \eqref{2.162} and \eqref{2.163}, for every $n\in\NN$
there exists $\tau_n\in(-\infty,\tau]$ such that
\begin{equation}
\label{2.164}
\lambda\gamma(\tau_n)-u(\tau_n)\leq\frac{1}{n}.
\end{equation}
However, from \eqref{2.160} and \eqref{2.161} it follows that
$\lambda\gamma-u$ is a nonincreasing function, which together with
\eqref{2.164} implies
\begin{equation}
\label{2.165}
\lambda\gamma(\tau)-u(\tau)\leq\frac{1}{n}\for n\in\NN.
\end{equation}
Now \eqref{2.159}, \eqref{2.160}, and \eqref{2.165} results in \eqref{2.158}.
\end{proof}

Now, from Lemma~\ref{l2.17} we get the following

\begin{lemma}
\label{l2.18}
Let $\ell_1\in \vtn$, and let there exist $\gamma\in\acloc$ such that
\eqref{2.4} holds.
Further, let $u\in\aclocrp$ satisfy
\begin{gather}
\label{2.195}
u'(t)\geq -\ell_1(u)(t)\forae t\leq t_0,\\
\label{2.166}
\sup\big\{u(t):t\leq t_0\big\}<+\infty.
\end{gather}
Then 
\begin{gather}
\label{2.167}
\ell_1(u)(t)\leq \frac{\ell_1(\gamma)(t)}{\gamma(t)}u(t)\forae t\leq t_0.
\end{gather}
\end{lemma}

\begin{proof}
Obviously, since $\ell_1\in \vtn$, according to
Lemma~\ref{l2.17} we have
$$
\sup\left\{\frac{u(t)}{\gamma(t)}:t\leq
  \tau\right\}=\frac{u(\tau)}{\gamma(\tau)}\for \tau\leq t_0.
$$
However, the latter means that the function $u/\gamma$ is
nondecreasing. Therefore, according to Lemma~\ref{l2.1} with
$\ell=\ell_1$, $\alpha=\vartheta(\gamma)$,
$\beta=\vartheta(u/\gamma)$, and $\vartheta$ given by \eqref{1.4},
we obtain \eqref{2.167}. 
\end{proof}

The other assertion which can be deduced from Lemma~\ref{l2.17} is the following

\begin{lemma}
\label{l2.19}
Let $\tau\in\RR$ and let there exist $\gamma\in\acloctau$ such that
\eqref{2.194} is fulfilled. 
Further, let $u\in\acloctaur$ satisfy \eqref{2.157},  
$$
\sup\big\{u(t):t\leq \tau\big\}<+\infty,
$$
and
\begin{equation}
\label{2.169}
u(\tau)\leq 0.
\end{equation}
Then 
\begin{equation}
\label{73.3}
u(t)\leq 0\for t\leq \tau.
\end{equation}
\end{lemma}

\begin{proof}
Assume on the contrary that there exists $\tau_0\in (-\infty,\tau)$ such that
\begin{equation}
\label{2.170}
u(\tau_0)>0.
\end{equation}
Then, according to Lemma~\ref{l2.17} we have \eqref{2.158}. However,
\eqref{2.158} together with \eqref{2.169} contradicts \eqref{2.170}.
\end{proof}

Analogously to Lemma~\ref{l2.6} one can prove the following

\begin{lemma}
\label{l2.20}
Let $p\in\lloctrp$, $\sigma\in\Sigma$, \eqref{2.1} hold, and let
\begin{equation}
\label{2.171}
\ell_0(1)(t)\geq p(t)\forae t\leq t_0.
\end{equation}
Let, moreover, $u\in\aclocrp$ satisfy
\begin{equation}
\label{2.172}
u'(t)\geq \ell_0(u)(t)-p(t)u(t)\forae t\leq t_0,
\end{equation}
and let there exist an interval $[\tau_0,\tau_1]\subset (-\infty,t_0]$
such that \eqref{2.84} is fulfilled. Then \eqref{2.85} holds.
\end{lemma} 

\begin{lemma}
\label{l2.21}
Let $p\in\lloctrp$, $\sigma\in\Sigma$, \eqref{2.1} and \eqref{2.171}
hold, and let 
\begin{equation}
\label{2.173}
\sup\left\{\int_{\sigma(t)}^t p(s)ds:t\leq t_0\right\}<+\infty.
\end{equation}
Let, moreover, $u\in\aclocrp$ satisfy \eqref{2.166} and \eqref{2.172}.
Then there exists a (finite) limit $u(-\infty)$.
\end{lemma}

\begin{proof}
Assume on the contrary that 
\begin{equation}
\label{2.174}
u^*-u_*>0
\end{equation}
where 
\begin{equation}
\label{2.175}
u_*=\liminf_{t\to -\infty}u(t),\qquad u^*=\limsup_{t\to -\infty}u(t).
\end{equation}
In view of \eqref{2.173} and \eqref{2.174} there exists $\delta>0$ such that 
\begin{equation}
\label{2.176}
2\delta<(u^*-u_*)e^{-M_{\sigma}}
\end{equation}
where
\begin{equation}
\label{2.177}
M_{\sigma}=\sup\left\{\int_{\sigma(t)}^t p(s)ds:t\leq t_0\right\}.
\end{equation}
Then, in view of \eqref{2.175}, there exists $t_{\delta}\leq t_0$ such
that 
\begin{equation}
\label{2.178}
u(t)\geq u_*-\delta\for t\leq t_{\delta}.
\end{equation}
Further, according to \eqref{2.175} there exist $\tau_0<\tau_1\leq t_{\delta}$
such that 
\begin{equation}
\label{2.179}
u(\tau_0)\geq u^*-\delta,\qquad u(\tau_1)\leq u_*+\delta,
\end{equation}
and, obviously, without loss of generality we can assume that
\eqref{2.84} holds. Thus, according to Lemma~\ref{l2.20} we have
\eqref{2.85}.

On the other hand, from \eqref{2.172} we get
$$
u(\tau_1)\geq
u(\tau_0)\exp\left(-\int_{\tau_0}^{\tau_1}p(s)ds\right)+
\int_{\tau_0}^{\tau_1}\ell_0(u)(s)\exp\left(-\int_s^{\tau_1}p(\xi)d\xi\right)ds,
$$
whence, on account of \eqref{2.1}, \eqref{2.171}, \eqref{2.178}, and \eqref{2.179} we
find
\begin{equation}
\label{2.180}
u_*+\delta\geq
(u^*-\delta)\exp\left(-\int_{\tau_0}^{\tau_1}p(s)ds\right)+(u_*-\delta)
\left(1-\exp\left(-\int_{\tau_0}^{\tau_1}p(s)ds\right)\right). 
\end{equation}
Now \eqref{2.180} results in 
$$
2\delta\geq (u^*-u_*)\exp\left(-\int_{\tau_0}^{\tau_1}p(s)ds\right)
$$
which, together with \eqref{2.85}, \eqref{2.174}, and \eqref{2.177}
contradicts \eqref{2.176}. 
\end{proof}

\begin{lemma}
\label{l2.22}
Let $\ell_1\in \vtn$, \eqref{2.29} hold, and let there exist
$\gamma\in\acloc$ satisfying \eqref{2.4}. 
Further, let $u\in\aclocrp$ satisfy 
\begin{equation}
\label{2.181}
u'(t)\geq \ell_0(u)(t)-\ell_1(u)(t)\forae t\leq t_0,
\end{equation}
and assume that there exists a finite limit $u(-\infty)$. Then 
\begin{equation}
\label{2.182}
u(t)\geq u(-\infty)\for t\leq t_0,
\end{equation}
and, in addition, if there exists $\tau\in(-\infty,t_0]$ such that
$u(\tau)=u(-\infty)$, then
\begin{equation}
\label{2.183}
u(t)=u(-\infty)\for t\leq \tau.
\end{equation}
\end{lemma}

\begin{proof}
To prove lemma it is sufficient to show that whenever there exists
$\tau\in(-\infty,t_0]$ such that
\begin{equation}
\label{2.184}
u(\tau)=\inf\big\{u(t):t\leq t_0\big\}
\end{equation}
then $u$ satisfies \eqref{2.183}, and so \eqref{2.182} holds
necessarily. Therefore, let $\tau\in(-\infty,t_0]$ be arbitrary but fixed such
that \eqref{2.184} holds. Put
\begin{equation}
\label{2.185}
z(t)=u(t)-u(\tau)\for t\leq \tau.
\end{equation}
Then, in view of \eqref{2.29}, \eqref{2.181}, \eqref{2.184}, and \eqref{2.185}, we
have
\begin{gather}
\label{2.186}
z(t)\geq 0\for t\leq \tau,\\
\label{2.187}
z'(t)\geq \ell_0(1)(t)u(\tau)-\ell_1(u)(t)\geq -\ell_1(z)(t)\mforae t\leq \tau,\qquad z(\tau)=0.
\end{gather}
Now, applying Lemma~\ref{l2.19}, on account of the inclusion
$\ell_1\in \vtn$, \eqref{2.187} yields 
\begin{equation}
\label{2.188}
z(t)\leq 0\for t\leq \tau.
\end{equation}
Thus \eqref{2.185}, \eqref{2.186}, and \eqref{2.188} implies \eqref{2.183}.
\end{proof}

\begin{lemma}
\label{l2.23}
Let $\ell_0,\ell_1\in \vtn$, $\omega\in\Sigma$, \eqref{2.29} and
\eqref{2.33} hold, and let there exist a function $\gamma\in\acloc$ satisfying
\eqref{2.4}. Let, moreover, $u\in\aclocrp$ satisfy \eqref{2.181}, and let there
exist a finite limit $u(-\infty)$. Then
\begin{equation}
\label{2.189}
u(-\infty)\limsup_{t\to
  -\infty}\int_{\omega(t)}^t\big[\ell_0(1)(s)-\ell_1(1)(s)\big]ds+\limsup_{t\to
-\infty}\int_{\omega(t)}^t q(s)ds=0
\end{equation}  
where 
\begin{equation}
\label{2.190}
q(t)\stackrel{def}{=}u'(t)-\ell_0(u)(t)+\ell_1(u)(t)\forae t\leq t_0.
\end{equation}
\end{lemma}

\begin{proof}
Assume on the contrary that \eqref{2.189} does not hold. Then, in view
of \eqref{2.29}, \eqref{2.181}, and \eqref{2.190} we have
$$
u(-\infty)\limsup_{t\to
  -\infty}\int_{\omega(t)}^t\big[\ell_0(1)(s)-\ell_1(1)(s)\big]ds+\limsup_{t\to
-\infty}\int_{\omega(t)}^t q(s)ds>0.
$$
Therefore, according to \eqref{2.33}, there exists $\delta>0$ such that 
\begin{equation}
\label{2.191}
u(-\infty)\limsup_{t\to -\infty}\int_{\omega(t)}^t\big[\ell_0(1)(s)-\ell_1(1)(s)\big]ds+\limsup_{t\to
-\infty}\int_{\omega(t)}^t q(s)ds>2\delta M_{\omega}
\end{equation}
where
$$
M_{\omega}=\sup\left\{\int_{\omega(t)}^t\ell_1(1)(s)ds:t\leq t_0\right\}.
$$
Further, note that according to Lemma~\ref{l2.22}, the inequality
\eqref{2.182} holds and, moreover, there exists $t_{\delta}\leq t_0$ such that
\begin{equation}
\label{2.192}
u(t)\leq u(-\infty)+\delta\for t\leq t_{\delta}.
\end{equation}
Now the integration of \eqref{2.190} from $\omega(t)$ to $t$ yields
$$
u(t)-u(\omega(t))=
\int_{\omega(t)}^t\big[\ell_0(u)(s)-\ell_1(u)(s)+q(s)\big]ds\for t\leq t_{\delta},
$$
whence, in view of \eqref{2.33}, \eqref{2.182}, and \eqref{2.192}, we
get
\begin{equation}
\label{2.193}
u(t)-u(\omega(t))\geq
u(-\infty)\int_{\omega(t)}^t\big[\ell_0(1)(s)-\ell_1(1)(s)\big]ds-\delta
M_{\omega}+\int_{\omega(t)}^t q(s)ds\mfor t\leq t_{\delta}.
\end{equation}
Now \eqref{2.193}, on account of \eqref{2.29}, \eqref{2.181}, and \eqref{2.190}, results in
\begin{gather*}
\delta M_{\omega}\geq u(-\infty)\limsup_{t\to -\infty}\int_{\omega(t)}^t\big[\ell_0(1)(s)-\ell_1(1)(s)\big]ds,\\
\delta M_{\omega}\geq\limsup_{t\to -\infty}\int_{\omega(t)}^t q(s)ds.
\end{gather*}
However, the latter inequalities contradict \eqref{2.191}.
\end{proof}

\section{Proofs}
\label{sec5}

\begin{proof}[Proof~of~Theorem~\ref{t2.1}]
Let $(a_n)_{n=1}^{+\infty}$ be a sequence of real numbers satisfying
\eqref{2.132}, and let $n\in\NN$ be arbitrary but fixed. Then,
according to Lemma~\ref{l2.7.1} we have
$$
-\ell_1\in\santt.
$$
Moreover, according to Lemma~\ref{l2.1} with $\ell=\ell_1$,
$\alpha=\vartheta(\gamma)$, $\beta=\vartheta(1/\gamma)$, and
$\vartheta$ given by \eqref{1.4}, from \eqref{2.5} it follows that
\eqref{2.29} is fulfilled. Finally, according to \eqref{2.130} and
\eqref{2.131}, there exists $q_0\in\lantrp$ such that 
$$
|\overline{f}(v)(t)|\leq q_0(t)\forae t\in [a_n,t_0],\quad v\in\cantr
$$
and, with respect to \eqref{2.2} and \eqref{2.3}, we have
\begin{gather}
\label{2.196}
\overline{f}(v)(t)\geq 0\forae t\in [a_n,t_0],\quad v\in\cantr,\\
\label{2.197}
\overline{f}(v)(t)=0\forae t\in [a_n,t_0],\quad -v\in\cantrp.
\end{gather}
Consequently, all the assumptions of Lemma~\ref{l2.13} (with
$f=\overline{f}$, $a=a_n$, $q=q_0$) are
fulfilled. Therefore, there exists a solution $u_n$ to the problem
\eqref{2.133}, \eqref{1.2} on $[a_n,t_0]$ satisfying
\begin{equation}
\label{2.198}
0\leq u_n(t)\for t\in[a_n,t_0].
\end{equation}
Furthermore, according to \eqref{2.196}, from \eqref{2.133} we obtain
$$
u'_n(t)\geq \ell_0(u_n)(t)-\ell_1(u_n)(t)\forae t\in[a_n,t_0].
$$
Thus, all the assumptions of Lemma~\ref{l2.8} (with $a=a_n$) are
fulfilled, and so  
\begin{equation}
\label{2.199}
u_n(t)\leq \kappa\for t\in[a_n,t_0].
\end{equation}
Now, \eqref{2.198} and \eqref{2.199} imply
\eqref{2.134}. Conseuqently, the theorem follows from
Lemmas~\ref{l2.15} and \ref{l3.3}.
\end{proof}

\begin{proof}[Proof~of~Theorem~\ref{t2.2}]
Let $(a_n)_{n=1}^{+\infty}$ be a sequence of real numbers satisfying
\eqref{2.132}, and let $n\in\NN$ be arbitrary but fixed. Then,
according to \eqref{2.130} and
\eqref{2.131}, with respect to \eqref{2.2} and \eqref{2.3}, we have
\eqref{2.196} and \eqref{2.197}.
Consequently, all the assumptions of Lemma~\ref{l2.14} (with
$f=\overline{f}$, $a=a_n$) are
fulfilled. Therefore, there exists a solution $u_n$ to the problem
\eqref{2.133}, \eqref{1.2} on $[a_n,t_0]$ satisfying \eqref{2.198} and
\begin{equation}
\label{*}
u'_n(t)\geq 0\forae t\in[a_n,t_0].
\end{equation}
Thus, \eqref{1.2}, \eqref{2.198}, and \eqref{*},
with respect to $c\in[0,\kappa]$, imply \eqref{2.134}. Conseuqently, the
theorem follows from Lemmas~\ref{l2.15} and \ref{l3.3}.  
\end{proof}

\begin{proof}[Proof~of~Theorem~\ref{t2.3}]
According to Theorem~\ref{t2.1}, there exists a global solution $u$ to the
problem \eqref{1.1}, \eqref{1.2} satisfying
\eqref{2.8}. We will show that $u$ is positive in $(-\infty,t_0]$. Assume on the contrary
that there exists $\tau<t_0$ such that $u(\tau)=0$. Then, in view of
$\ell_0,f\in\vtn$, \eqref{2.3}, and \eqref{2.8}, from \eqref{1.1} we obtain \eqref{2.157}.

On the other hand, in view of the inclusion $\ell_1\in \vtn$, from
\eqref{2.4} it follows that \eqref{2.194} holds.
Therefore, according to Lemma~\ref{l2.19}, on account of \eqref{2.8}
we get \eqref{2.139}. Now Lemma~\ref{l2.16} yields that $u\equiv 0$ on
$(-\infty,t_0]$ which, together with $c>0$, contradicts \eqref{1.2}. 
\end{proof}

\begin{proof}[Proof~of~Theorem~\ref{t2.4}]
According to Theorem~\ref{t2.2}, there exists a global solution $u$ to the
problem \eqref{1.1}, \eqref{1.2} satisfying
\eqref{2.8} and \eqref{2.9}. We will show that $u$ is positive in $(-\infty,t_0]$. Assume
on the contrary that there exists $\tau<t_0$ such that
$u(\tau)=0$. Then from \eqref{2.8} and \eqref{2.9} we get
\eqref{2.139}. Now Lemma~\ref{l2.16} yields that $u\equiv 0$ on $(-\infty,t_0]$ which,
together with $c>0$, contradicts \eqref{1.2}.
\end{proof}

\begin{proof}[Proof~of~Theorem~\ref{t2.1.r}]
First note that, according to the inclusion $\ell_1\in V$, from
\eqref{2.4.r} it follows that \eqref{2.4} holds. Therefore, according
to Theorem~\ref{t2.1}, there exists a global solution $u$ to the problem
\eqref{1.1}, \eqref{1.2} satisfying \eqref{2.8}. We will show
that $u$ is positive in $[t_0,+\infty)$. Assume on the contrary that
\eqref{2.8.r} does not hold. Then, in view of \eqref{1.2}, there
exists $\tau>t_0$ such that $u(\tau)=0$ and
\begin{equation}
\label{73.2}
u(t)>0\for t\in[t_0,\tau).
\end{equation}
Now, on account of \eqref{2.8}, \eqref{2.3.5.r}, \eqref{73.2}, and the
inclusions $\ell_0,f\in V$, from \eqref{1.1} we obtain
\eqref{2.157}. Moreover, from \eqref{2.4.r}, with respect to the
inclusion $\ell_1\in V$, the inequality \eqref{2.194}
follows. Therefore, according to Lemma~\ref{l2.19} we have
\eqref{73.3}. However, \eqref{73.3} contradicts \eqref{73.2}.
\end{proof}

\begin{proof}[Proof~of~Theorem~\ref{t2.2.r}]
Put
\begin{equation}
\label{73.4}
\widetilde{f}(v)(t)\stackrel{def}{=}f(|v|)(t)\forae t\in\RR,\quad v\in\cloc
\end{equation}
and consider the auxiliary equation
\begin{equation}
\label{1.1.aux}
u'(t)=\ell_0(u)(t)-\ell_1(u)(t)+\widetilde{f}(u)(t).
\end{equation}
Note that from \eqref{2.3.5.rr} it follows that 
$$
\widetilde{f}(v)(t)\sgn v(t)\leq q(t,\|v\|) \mforae t\geq t_0,\quad
v\in\cnul,\quad 
-\kappa\leq v(t)\leq \kappa\mfor t\leq t_0
$$
holds. Therefore, according to Theorem~\ref{t2.2}, for every
$c\in[0,\kappa]$ there exists a global solution $u$ to the problem
\eqref{1.1.aux}, \eqref{1.2} satisfying \eqref{2.8} and
\eqref{2.9}.

Put
\begin{equation}
\label{73.6}
w(t)=\sup\big\{u(s):s\leq t\big\}\for t\in\RR
\end{equation}
and 
$$
A=\big\{t\in\RR:w(t)=u(t)\big\}.
$$
Obviously, on account of \eqref{73.6} and \eqref{2.9} we have
\begin{gather}
\label{73.7}
w\in\aclocr,\\
\label{73.8}
w(t)\geq 0 \for t\in\RR,\\
\label{73.9}
w'(t)\geq 0 \forae t\in\RR,\\
\label{73.10}
w(t)=u(t) \for t\leq t_0,\\
\label{73.11}
w(t)\geq u(t) \for t\geq t_0,
\end{gather}
and 
\begin{equation}
\label{73.12}
w'(t)=
\begin{cases}
u'(t)&\casforae t\in A,\\
0&\casforae t\in\RR\setminus A.
\end{cases}
\end{equation}
According to \eqref{73.10} and \eqref{73.11}, from \eqref{1.1.aux} it
follows that 
\begin{equation}
\label{73.13}
u'(t)\leq \ell_0(w)(t)-\ell_1(u)(t)+\widetilde{f}(u)(t)\forae t\in\RR.
\end{equation}
On the other hand, in view of the inclusions $\ell_0-\ell_1\in\pplus$
and $f\in V$, on account of \eqref{2.3.5.rr}, \eqref{73.4}, and
\eqref{73.7}--\eqref{73.11}, we have
\begin{equation}
\label{73.14}
\ell_0(w)(t)-\ell_1(u)(t)+\widetilde{f}(u)(t)\geq
\ell_0(w)(t)-\ell_1(w)(t)\geq 0\forae t\in\RR.
\end{equation}
Now from \eqref{73.12}--\eqref{73.14} we get
\begin{equation}
\label{73.15}
w'(t)\leq \ell_0(w)(t)-\ell_1(u)(t)+\widetilde{f}(u)(t)\forae t\in\RR.
\end{equation}
Put
\begin{equation}
\label{73.16}
z(t)=w(t)-u(t)\for t\in\RR.
\end{equation}
Then, in view of \eqref{1.1.aux}, \eqref{73.10}, \eqref{73.15}, and
\eqref{73.16}, we have
\begin{gather}
\label{73.17}
z'(t)\leq \ell_0(z)(t)\forae t\in\RR,\\
\label{73.18}
z(t)=0\for t\leq t_0.
\end{gather}
Now the inclusion $\ell_0\in V$, according to Proposition~\ref{p2.2},
yields $\ell_0\in\mathcal{S}_{t_0\tau}(t_0)$ for every
$\tau>t_0$. Consequently, \eqref{73.17} and \eqref{73.18} result in
$z(t)\leq 0$ for $t\geq t_0$ whence, on account of \eqref{73.16}, we
get
\begin{equation}
\label{73.19}
w(t)\leq u(t)\for t\geq t_0.
\end{equation}
However, \eqref{73.10}, \eqref{73.11}, and \eqref{73.19} yield that
$w\equiv u$ on $\RR$, and, consequently, on account of \eqref{73.4},
\eqref{1.1.aux}, \eqref{73.8}, and \eqref{73.9}, we have that $u$ is a
global solution also to the problem \eqref{1.1}, \eqref{1.2}
satisfying \eqref{2.8} and \eqref{2.9.r}.
\end{proof}

\begin{proof}[Proof~of~Theorem~\ref{t2.9}]
Let $u$ be a solution to \eqref{1.1} on $(-\infty,t_0]$ satisfying
\eqref{2.8}, and let $\tau\in\RR$, $\tau\leq t_0$, be arbitrary but
fixed. Then, according to \eqref{2.23}, $u$ satisfies also \eqref{2.156} and
\eqref{2.157}. Moreover, since $\ell_1\in \vtn$, the inequality
\eqref{2.194} holds. Thus, according to Lemma~\ref{l2.17} we have
\eqref{2.158}, and so the function $u/\gamma$
is nondecreasing in $(-\infty,t_0]$. Consequently, in view of
\eqref{2.8}, there exists a finite limit 
\begin{equation}
\label{2.216}
0\leq \lim_{t\to -\infty}\frac{u(t)}{\gamma(t)}<+\infty.
\end{equation}
Now from \eqref{2.24} and \eqref{2.216} it follows that there exists a
finite limit $u(-\infty)$.
\end{proof}

\begin{proof}[Proof~of~Theorem~\ref{t2.10}]
Let $u$ be a solution to \eqref{1.1} on $(-\infty,t_0]$ satisfying
\eqref{2.8}. Then, in view of \eqref{2.3}, $u$ satisfies also
\eqref{2.195}. Thus, according to Lemma~\ref{l2.18}, the estimate
\eqref{2.167} holds. Therefore, in view of \eqref{2.3} and \eqref{2.8},
we have \eqref{2.181} whence, on account of \eqref{2.167} we get 
\begin{equation}
\label{2.220}
u'(t)\geq \ell_0(u)(t)-\frac{\ell_1(\gamma)(t)}{\gamma(t)}u(t)\forae
t\leq t_0.
\end{equation}
Now, \eqref{2.1}, \eqref{2.5}, \eqref{2.6}, \eqref{2.8}, and
\eqref{2.220} yield that all the assumptions of Lemma~\ref{l2.21} are
fulfilled with $p=\ell_1(\gamma)/\gamma$. Therefore, there
exists a finite limit $u(-\infty)$.
\end{proof}

\begin{proof}[Proof~of~Theorem~\ref{t2.11}]
First note that according to Remark~\ref{r2.6} we have
\eqref{2.25}. Further, \eqref{2.3} and
\eqref{2.8} yields \eqref{2.181}. Therefore, according to
Lemma~\ref{l2.23}, on account of \eqref{2.25}, we have
\begin{equation}
\label{2.221}
u(-\infty)\limsup_{t\to -\infty}\int_{\omega(t)}^t\ell_0(1)(s)ds+\limsup_{t\to
-\infty}\int_{\omega(t)}^t f(u)(s)ds=0
\end{equation}
for any $\omega\in\Sigma$.

Consequently, if \eqref{2.30} holds then we define $\omega$ in the following way:
let the values $\omega(t_n)$ and $\omega(t_0)$ be defined by
\begin{equation}
\label{2.222}
\int_{\omega(t_{n-1})}^{t_{n-1}}\ell_0(1)(s)ds=1\for n\in\NN,
\end{equation}
where
\begin{equation}
\label{2.223}
t_n=t_0-n\for n\in\NN.
\end{equation}
Further, put
$$
\omega(t)=(\omega(t_{n-1})-\omega(t_n))(t-t_n)+\omega(t_n)\for
t\in (t_n,t_{n-1}),\quad n\in\NN
$$
and $\omega(t)\stackrel{def}{=}\omega(t_0)$ for $t>t_0$.
Then, obviously, $\omega\in\Sigma$ and, in view of \eqref{2.222} and \eqref{2.223}, we have
$$
\limsup_{t\to -\infty}\int_{\omega(t)}^t\ell_0(1)(s)ds>0.
$$
Thus, in view of \eqref{2.3} and
\eqref{2.8}, from \eqref{2.221} it follows that \eqref{2.32} is fulfilled.

Further note that, in view of Lemma~\ref{l2.22} we have
\eqref{2.182} and if $c=u(-\infty)$ then $u(t)=c$ for $t\leq
t_0$. Therefore, if \eqref{2.31} holds then, on account of 
\eqref{2.8}, we get \eqref{2.32} again.
\end{proof}

\begin{proof}[Proof~of~Theorem~\ref{t2.12}]
First note that \eqref{2.3} and
\eqref{2.8} yields \eqref{2.181}. Therefore, according to
Lemma~\ref{l2.23}, we have
\begin{equation}
\label{2.224}
u(-\infty)\limsup_{t\to -\infty}\int_{\omega(t)}^t\big[\ell_0(1)(s)-\ell_1(1)(s)\big]ds+\limsup_{t\to
-\infty}\int_{\omega(t)}^t f(u)(s)ds=0.
\end{equation}
Consequently, if \eqref{2.34} holds, then, in view of \eqref{2.3} and
\eqref{2.8}, from \eqref{2.224} it follows that \eqref{2.32} is fulfilled.

Further note that, in view of Lemma~\ref{l2.22} we have
\eqref{2.182} and if $c=u(-\infty)$ then $u(t)=c$ for $t\leq
t_0$. Therefore, if \eqref{2.31} holds then, on account of 
\eqref{2.8}, we get \eqref{2.32} again.
\end{proof}

\begin{proof}[Proof~of~Proposition~\ref{p2.1}]
Let $c\in(0,\kappa)$ be arbitrary but fixed and let $u\in\cnulok$ satisfy
conditions of Definition~\ref{d2.1} with $\tau=t_0$. We will show that 
\begin{equation}
\label{2.28}
\limsup_{t\to -\infty}\int_{\omega(t)}^t f(u)(s)ds>0.
\end{equation}
Obviously, for every $n\in\NN$ there exists $t_n\leq t_0$ such that
$$
u(-\infty)\leq u(t)\leq u(-\infty)+\frac{c-u(-\infty)}{n}\for t\leq t_n.
$$
Put
$$
\vartheta_n(u)(t)=
\begin{cases}
u(t)&\casif t\leq t_n,\\
u(t_n)&\casif t>t_n
\end{cases}
\for n\in\NN.
$$
Then
\begin{gather}
\label{2.225}
0<\vartheta_n(u)(t)\leq c\for t\in\RR,\quad n\in\NN,\\
\label{2.226}
\|\vartheta_n(u)-u(-\infty)\|\leq \frac{c-u(-\infty)}{n}\for n\in\NN,
\end{gather}
and, on account of the inclusion $f\in \vtn$, for
every $n\in\NN$ we have
\begin{equation}
\label{2.227}
\int_{\omega(t)}^t f(u)(s)ds=\int_{\omega(t)}^t
f(\vartheta_n(u))(s)ds\for t\leq t_n.
\end{equation}
On the other hand, in view of \eqref{2.226} and the continuity of $h_1$,
there exist $\varepsilon_n>0$ $(n\in\NN)$ such that 
\begin{equation}
\label{2.228}
\lim_{n\to +\infty}\varepsilon_n=0
\end{equation} 
and
\begin{equation}
\label{2.229}
\|h_1(\vartheta_n(u))-h_1(u(-\infty))\|_{\infty}\leq \varepsilon_n\for n\in\NN. 
\end{equation} 
Now, \eqref{2.35}, \eqref{2.225}, \eqref{2.227}, and \eqref{2.229}
results in
\begin{equation}
\label{2.230}
\int_{\omega(t)}^t f(u)(s)ds\geq \int_{\omega(t)}^t
g(s)h_1(u(-\infty))(s)ds-\varepsilon_n g^*\for t\leq t_n,\quad n\in\NN
\end{equation}
where
\begin{equation}
\label{2.231}
g^*=\sup\left\{\int_{\omega(t)}^t |g(s)|ds:t\leq t_0\right\}.
\end{equation}
Consequently, from \eqref{2.230} it follows that 
\begin{equation}
\label{2.232}
\limsup_{t\to -\infty}\int_{\omega(t)}^t f(u)(s)ds\geq \limsup_{t\to -\infty}\int_{\omega(t)}^t
g(s)h_1(u(-\infty))(s)ds-\varepsilon_n g^*\mfor n\in\NN.
\end{equation}
Thus \eqref{2.232}, on account of \eqref{2.36}, \eqref{2.37},
\eqref{2.228}, and \eqref{2.231}, yields \eqref{2.28}.
\end{proof}

\begin{proof}[Proof~of~Corollary~\ref{c2.1}]
Let $u$ be a solution to the problem \eqref{1.1}, \eqref{1.2} on
$(-\infty,t_0]$ satisfying \eqref{2.8}. Then, in view of \eqref{2.35}
and \eqref{2.217}, all the
conditions of Theorem~\ref{t2.9} are satisfied. Therefore, there
exists a finite limit $u(-\infty)$.

Now we will show that \eqref{2.38} and \eqref{2.39} imply \eqref{2.36}
and \eqref{2.37} with a suitable function $\omega$. Let
\begin{equation}
\label{2.233}
\varphi(t)=\frac{1}{(t_0+1-t)^2}\for t\leq t_0.
\end{equation}
Then, obviously,
\begin{equation}
\label{2.234}
\varphi(t)>0 \for t\leq t_0,\qquad \lim_{t\to -\infty}\int_t^{t_0}\varphi(s)ds=1.
\end{equation}
Define $\omega$ by
\begin{equation}
\label{2.235}
\int_{\omega(t)}^{t_0}\big(g(s)+\varphi(s)\big)ds=1+\int_t^{t_0}\big(g(s)+\varphi(s)\big)ds\for
t\leq t_0
\end{equation}
and $\omega(t)\stackrel{def}{=}\omega(t_0)$ for $t>t_0$.
Then, in view of \eqref{2.234} and the non-negativity of $g$, we have
$\omega\in\Sigma$. Moreover, \eqref{2.234} and \eqref{2.235}
yields 
\begin{gather}
\label{2.236}
\lim_{t\to -\infty}\int_{\omega(t)}^t\varphi(s)ds=0,\\
\label{2.237}
\int_{\omega(t)}^t g(s)ds=1-\int_{\omega(t)}^t\varphi(s)ds\for t\leq t_0.
\end{gather}
Therefore, from \eqref{2.236} and \eqref{2.237} we get \eqref{2.36} and
\begin{equation}
\label{2.238}
\lim_{t\to -\infty}\int_{\omega(t)}^t g(s)ds=1.
\end{equation}
Now \eqref{2.37} follows from \eqref{2.39} and \eqref{2.238}.

Consequently, according to Proposition~\ref{p2.1}, all the
assumptions of Theorem~\ref{t2.11} hold and so \eqref{2.32} is satisfied.
\end{proof}

\begin{proof}[Proof~of~Corollary~\ref{c2.2}]
Let $u$ be a solution to the problem \eqref{1.1}, \eqref{1.2} on
$(-\infty,t_0]$ satisfying \eqref{2.8}.  Then, in view of \eqref{2.35}
and \eqref{2.217}, all the conditions of Theorem~\ref{t2.10} are
satisfied. Therefore, there exists a finite limit $u(-\infty)$.

Further, \eqref{2.1} implies $\ell_0\in \vtn$, \eqref{2.4} implies
that $\gamma$ is a nonincreasing function, and so, according to
Lemma~\ref{l2.1} with $\ell=\ell_1$, $\alpha=\vartheta(\gamma)$,
$\beta=\vartheta(1/\gamma)$, and $\vartheta$ given by \eqref{1.4},
from \eqref{2.5} we get \eqref{2.29}. Finally, according to
Proposition~\ref{p2.1}, the inclusion \eqref{2.31} holds. Consequently, all the
assumptions of Theorem~\ref{t2.12} are fulfilled and so \eqref{2.32} is satisfied.
\end{proof}

\begin{proof}[Proof~of~Theorem~\ref{t2.5}]
Define operators $\ell_0$, $\ell_1$, and $f$ by
\begin{gather}
\label{2.201}
\ell_i(u)(t)=p_i(t)u(\mu_i(t))\forae t\in\RR\quad (i=0,1),\\
\label{2.202}
f(u)(t)=h(t,u(t),u(\nu(t)))\forae t\in\RR.
\end{gather}
Then, in view of \eqref{2.15.r} we have $\ell_0,\ell_1,f\in V$, and so from
\eqref{2.13} and \eqref{2.14} it follows that \eqref{2.2} and
\eqref{2.3} hold.  

Further, put
\begin{equation}
\label{2.203}
\gamma(t)=\exp\left(e\int_{t}^{t_0} p_1(s)ds\right)\for t\leq t_0.
\end{equation}
Then, in view of \eqref{2.16}, we have 
\begin{equation}
\label{2.204}
\gamma(t)=\gamma(\mu_1(t))\exp\left(-e\int_{\mu_1(t)}^t
  p_1(s)ds\right)\geq \frac{\gamma(\mu_1(t))}{e}\forae t\leq t_0,
\end{equation}
and so
\begin{equation}
\label{2.205}
\gamma'(t)=-ep_1(t)\gamma(t)\leq-p_1(t)\gamma(\mu_1(t))\forae t\leq t_0.
\end{equation}
Consequently, \eqref{2.4} holds. Moreover, from
\eqref{2.17}, on account of \eqref{2.203}, we get \eqref{2.5}. Finally,
from \eqref{2.18}, in view of \eqref{2.16}, we obtain
\begin{equation}
\label{2.206}
M_{\mu}<+\infty.
\end{equation} 
Now, let $c\in[0,\kappa e^{-M_{\mu}})$. Then there exists $\varepsilon>0$
such that 
\begin{equation}
\label{2.207}
c\leq \kappa e^{-(M_{\mu}+\varepsilon)}.
\end{equation}
Put
$$
\varphi(t)=\frac{\varepsilon}{(t_0+1-t)^2}\for t\leq t_0.
$$
Then, obviously,
\begin{equation}
\label{2.208}
\varphi(t)>0 \for t\leq t_0,\qquad \lim_{t\to -\infty}\int_t^{t_0}\varphi(s)ds=\varepsilon.
\end{equation}
Define a function $\sigma:\RR\to\RR$ by the equalities
\begin{equation}
\label{2.209}
\int_{\sigma(t)}^{t_0}\big(P_1(s)+\varphi(s)\big)ds=M_{\mu}+\varepsilon+\int_t^{t_0}\big(P_1(s)+\varphi(s)\big)ds\for
t\leq t_0
\end{equation}
and $\sigma(t)\stackrel{def}{=}\sigma(t_0)$ for $t>t_0$, where
\begin{equation}
\label{2.210}
P_1(t)=p_1(t)\exp\left(e\int_{\mu_1(t)}^tp_1(s)ds\right)\forae t\leq t_0.
\end{equation}
Then, in view of \eqref{2.208}, \eqref{2.210}, and the non-negativity of $p_1$, we have
$\sigma\in\Sigma$. Moreover, \eqref{2.206} and \eqref{2.208}--\eqref{2.210} yield
\begin{equation}
\label{2.211}
\int_{\sigma(t)}^t p_1(s)\exp\left(e\int_{\mu_1(s)}^sp_1(\xi)d\xi\right)ds=
M_{\mu}+\varepsilon-\int_{\sigma(t)}^t \varphi(s)ds\for t\leq t_0
\end{equation}
and
\begin{equation}
\label{2.212}
\sup\left\{\int_{\sigma(t)}^t
  p_1(s)\exp\left(e\int_{\mu_1(s)}^sp_1(\xi)d\xi\right)ds:t\leq t_0\right\}=
M_{\mu}+\varepsilon<+\infty.
\end{equation}
Now, from \eqref{2.19} and \eqref{2.211}, on account of \eqref{2.208}, we get
$$
\int_{\sigma(t)}^t
p_1(s)\exp\left(e\int_{\mu_1(s)}^sp_1(\xi)d\xi\right)ds>
\int_{\mu_0(t)}^t
p_1(s)\exp\left(e\int_{\mu_1(s)}^sp_1(\xi)d\xi\right)ds\mforae t\leq t_0
$$
whence, in view of the non-negativity of $p_1$, we obtain
\begin{equation}
\label{2.213}
\sigma(t)\leq \mu_0(t)\forae t\leq t_0.
\end{equation}
Therefore, \eqref{2.212} and \eqref{2.213}, with respect
to \eqref{2.15.r}, \eqref{2.203}, and \eqref{2.207}, imply \eqref{2.1},
\eqref{2.6}, and $c\in[0,\kappa e^{-M_{\sigma}}]$ with $M_{\sigma}$ defined by
\eqref{2.7}. 

Thus, the assertion follows from Theorem~\ref{t2.1}.
\end{proof}

\begin{proof}[Proof~of~Theorem~\ref{t2.6}]
Define operators $\ell_0$, $\ell_1$, and $f$ by
\eqref{2.201} and \eqref{2.202}. Then, in view of \eqref{2.15.r} we have
$\ell_0,\ell_1,f\in V$, and so from
\eqref{2.13} and \eqref{2.14} it follows that \eqref{2.2} and
\eqref{2.3} hold. We will show that \eqref{2.20} and \eqref{2.21}
imply $\ell_0-\ell_1\in\ptnplus$. Indeed, let $u\in\aclocrp$ be a
nondecreasing function. Then, in view of \eqref{2.21}, we have
\begin{equation}
\label{2.214}
p_1(t)\big(u(\mu_0(t))-u(\mu_1(t))\big)\geq 0\forae t\leq t_0.
\end{equation}
On the other hand, on account of \eqref{2.201}, we find
\begin{equation}
\label{2.215}
\ell_0(u)(t)-\ell_1(u)(t)=\big(p_0(t)-p_1(t)\big)u(\mu_0(t))+
p_1(t)\big(u(\mu_0(t))-u(\mu_1(t))\big)\mforae t\leq t_0.
\end{equation}
Thus, from \eqref{2.215}, in view of \eqref{2.20} and \eqref{2.214},
we obtain 
$$
\ell_0(u)(t)-\ell_1(u)(t)\geq 0\forae t\leq t_0.
$$
Consequently, $\ell_0-\ell_1\in\ptnplus$, and the assertion follows from
Theorem~\ref{t2.2}. 
\end{proof}

\begin{proof}[Proof~of~Theorem~\ref{t2.7}]
Let $c\in \left(0,\kappa e^{-M_{\mu}}\right)$ be arbitrary but fixed.
According to Theorem~\ref{t2.5}, there exists a global solution $u$ to the
problem \eqref{1.3}, \eqref{1.2} satisfying
\eqref{2.8}. We will show that $u$ is positive in $(-\infty,t_0]$. Assume on the contrary
that there exists $\tau<t_0$ such that $u(\tau)=0$. Define operators
$\ell_0$, $\ell_1$, and $f$ by \eqref{2.201} and \eqref{2.202}. Then,
in view of \eqref{2.15.r} we have $\ell_0,\ell_1,f\in V$ and
\eqref{2.157} is fulfilled.  

Further, define $\gamma$ by \eqref{2.203}. Then, in view of
\eqref{2.16}, we have \eqref{2.204} and \eqref{2.205}. 
Consequently, \eqref{2.194} holds. Therefore, according to
Lemma~\ref{l2.19}, on account of \eqref{2.8}, we get \eqref{2.139}. 
Now Lemma~\ref{l2.16} yields that $u\equiv 0$ on $(-\infty,t_0]$ which,
together with $c>0$, contradicts \eqref{1.2}.
\end{proof}

\begin{proof}[Proof~of~Theorem~\ref{t2.8}]
Let $c\in \left(0,\kappa e^{-M_{\mu}}\right)$ be arbitrary but fixed.
According to Theorem~\ref{t2.6}, there exists a global solution $u$ to the
problem \eqref{1.3}, \eqref{1.2} satisfying
\eqref{2.8} and \eqref{2.9}. We will show that $u$ is positive in
$(-\infty,t_0]$. Assume
on the contrary that there exists $\tau<t_0$ such that
$u(\tau)=0$. Then from \eqref{2.8} and \eqref{2.9} we get
\eqref{2.139}. 
Define operators
$\ell_0$, $\ell_1$, and $f$ by \eqref{2.201} and \eqref{2.202}. Then,
in view of \eqref{2.15.r} we have $\ell_0,\ell_1,f\in V$.
Now Lemma~\ref{l2.16} yields that $u\equiv 0$ on 
$(-\infty,t_0]$ which, together with $c>0$, contradicts \eqref{1.2}.
\end{proof}

\begin{proof}[Proof~of~Theorem~\ref{t2.5.r}]
Analogously to the proof of Theorem~\ref{t2.5} one can show that all
the assumptions of Theorem~\ref{t2.1.r} are fulfilled.
\end{proof}

\begin{proof}[Proof~of~Theorem~\ref{t2.6.r}]
Analogously to the proof of Theorem~\ref{t2.6} one can show that all
the assumptions of Theorem~\ref{t2.2.r} are fulfilled.
\end{proof}

\begin{proof}[Proof~of~Theorem~\ref{t2.13}]
Define operators $\ell_0$, $\ell_1$ and $f$ by \eqref{2.201} and
\eqref{2.202}. Then, in view of \eqref{2.15} we have
$\ell_0,\ell_1,f\in \vtn$, and so from \eqref{2.13} it follows that
\eqref{2.23} holds.  

Furthermore, define $\gamma$ by \eqref{2.203}. Then, on account of
\eqref{2.16}, we have \eqref{2.204} and \eqref{2.205}. Therefore,
\eqref{2.4} holds and \eqref{2.40} implies \eqref{2.24}. Consequently,
the assertion follows from Theorem~\ref{t2.9}. 
\end{proof}

\begin{proof}[Proof~of~Theorem~\ref{t2.14}]
Define operators $\ell_0$, $\ell_1$, and $f$ by
\eqref{2.201} and \eqref{2.202}. Then, in view of \eqref{2.15} we have
$\ell_0,\ell_1,f\in \vtn$, and so from \eqref{2.13} it follows
that \eqref{2.3} holds. 

Further, define $\gamma$ by \eqref{2.203}. Then, in view of
\eqref{2.16}, we have \eqref{2.204} and \eqref{2.205}. 
Consequently, \eqref{2.4} holds. Moreover, from \eqref{2.17}, on
account of \eqref{2.203}, we get \eqref{2.5}.

Now, let
\begin{equation}
\label{2.239}
p^*=\esssup\left\{\int_{\mu_0(t)}^t p_1(s)ds:t\leq t_0\right\},
\end{equation}
and let $\varphi$ be given by \eqref{2.233}. Then \eqref{2.234}
holds. Define a function $\sigma:\RR\to\RR$ by the equalities
\begin{equation}
\label{2.240}
\int_{\sigma(t)}^{t_0}\big(p_1(s)+\varphi(s)\big)ds=p^*+1+\int_t^{t_0}\big(p_1(s)+\varphi(s)\big)ds\for
t\leq t_0
\end{equation}
and $\sigma(t)\stackrel{def}{=}\sigma(t_0)$ for $t>t_0$.
Then, in view of \eqref{2.234} and the non-negativity of $p_1$, we
have $\sigma\in\Sigma$. Moreover, \eqref{2.234}
and \eqref{2.240} yields
\begin{equation}
\label{2.241}
\int_{\sigma(t)}^t p_1(s)ds=p^*+1-\int_{\sigma(t)}^t\varphi(s)ds> p^*\for t\leq t_0
\end{equation}
and
\begin{equation}
\label{2.242}
\sup\left\{\int_{\sigma(t)}^t p_1(s)ds: t\leq t_0\right\}<+\infty.
\end{equation}
Now, from \eqref{2.239} and \eqref{2.241} we get
$$
\int_{\sigma(t)}^t p_1(s)ds> \int_{\mu_0(t)}^t p_1(s)ds\forae
t\leq t_0
$$
whence, in view of non-negativity of $p_1$, we obtain \eqref{2.213}.
Therefore, \eqref{2.213} and \eqref{2.242}, with respect to \eqref{2.15},
\eqref{2.16}, and \eqref{2.203}, implies \eqref{2.1} and \eqref{2.6}.

Thus, the assertion follows from Theorem~\ref{t2.10}.
\end{proof}

\begin{proof}[Proof~of~Theorem~\ref{t2.15}]
Assume on the contrary that $u(-\infty)\in (0,\kappa)$. Then, in view of
\eqref{2.129}, there exist $\delta>0$ and $t_{\delta}\leq t_0$ such that
\begin{gather}
\label{2.243}
h_1(x,y)>0\for x,y\in[u(-\infty)-\delta, u(-\infty)+\delta],\\
\label{2.244}
0<u(-\infty)-\delta\leq u(t)\leq u(-\infty)+\delta<\kappa\for t\leq t_{\delta}.
\end{gather}
Integrating \eqref{1.3} from $t$ to $t_{\delta}$ we obtain
$$
u(t_{\delta})-u(t)=\int_t^{t_{\delta}}
\big[p_0(s)u(\mu_0(s))-p_1(s)u(\mu_1(s))+h(s,u(s),u(\nu(s)))\big]ds\for
t\leq t_{\delta}
$$
whence, on account of \eqref{2.15}, \eqref{2.41}, and \eqref{2.244} we
get
\begin{multline}
\label{2.245}
u(t_{\delta})-u(t)\geq (u(-\infty)-\delta)\int_t^{t_{\delta}}
p_0(s)ds-\kappa\int_t^{t_{\delta}}p_1(s)ds\\
+\int_t^{t_{\delta}}g(s)h_1(u(s),u(\nu(s)))ds\mfor
t\leq t_{\delta}.
\end{multline}
On the other hand, in view of \eqref{2.15}, \eqref{2.243}, and \eqref{2.244} we have
\begin{equation}
\label{2.168}
h_1(u(t),u(\nu(t)))\geq h_*>0\forae t\leq t_{\delta}
\end{equation}
where
\begin{equation}
\label{2.246}
h_*=\min\big\{h_1(x,y):u(-\infty)-\delta\leq x,y\leq u(-\infty)+\delta\big\}.
\end{equation}
Therefore, if \eqref{2.38} or \eqref{2.42} is fulfilled then, on account
of \eqref{2.40}, \eqref{2.244}, and \eqref{2.168}, from \eqref{2.245} we
obtain $u(-\infty)=-\infty$, a contradiction.
\end{proof}

\begin{proof}[Proof~of~Theorem~\ref{t2.16}]
Assume on the contrary that $u(-\infty)\in (0,\kappa)$. Then, in view of
\eqref{2.129}, there exist $\delta>0$ and $t_{\delta}\leq t_0$ such that
\eqref{2.243} and \eqref{2.244} hold. Therefore, on account of
\eqref{2.15}, \eqref{2.243}, and \eqref{2.244}, we have \eqref{2.168}
where $h_*$ is given by \eqref{2.246}. Moreover, from \eqref{1.3},
with respect to \eqref{2.15}, \eqref{2.41}, and \eqref{2.168}, we get
\begin{equation}
\label{2.247}
u'(t)\geq p_0(t)u(\mu_0(t))-p_1(t)u(\mu_1(t))\forae t\leq t_{\delta}.
\end{equation}
Define operators $\ell_0$ and $\ell_1$ by \eqref{2.201}. Then, in view
of \eqref{2.15}, \eqref{2.20}, \eqref{2.44}, and \eqref{2.247} we have
$\ell_0,\ell_1\in\vtd$,
\begin{gather*}
\ell_0(1)(t)\geq \ell_1(1)(t)\forae t\leq t_{\delta},\\
\sup\left\{\int_{\omega(t)}^t\ell_1(1)(s)ds:t\leq t_{\delta}\right\}<+\infty,\\
u'(t)\geq \ell_0(u)(t)-\ell_1(u)(t)\forae t\leq t_{\delta}.
\end{gather*}
Further, define $\gamma$ by \eqref{2.203}. Then, on account of
\eqref{2.16}, we have \eqref{2.204} and \eqref{2.205}. 
Consequently, 
$$
\gamma'(t)\leq -\ell_1(\gamma)(t)\forae t\leq t_{\delta},
$$
Therefore, all the assumptions of Lemma~\ref{l2.23} (with
$t_0=t_{\delta}$) are fulfilled, and thus
\begin{equation}
\label{2.248}
u(-\infty)\limsup_{t\to
  -\infty}\int_{\omega(t)}^t\big[p_0(s)-p_1(s)\big]ds+\limsup_{t\to
-\infty}\int_{\omega(t)}^t h(s,u(s),u(\nu(s)))ds=0.
\end{equation}  
However, from \eqref{2.41}, \eqref{2.168}, and \eqref{2.248}, with
respect to \eqref{2.20}, the inclusions $u(-\infty)\in(0,\kappa)$,
$\omega\in\Sigma$, and the non-negativity of $g$, it follows that
\begin{equation}
\label{2.249}
u(-\infty)\limsup_{t\to
  -\infty}\int_{\omega(t)}^t\big[p_0(s)-p_1(s)\big]ds+h_*\limsup_{t\to
-\infty}\int_{\omega(t)}^t g(s)ds=0.
\end{equation}  
Now it is clear that each of the conditions \eqref{2.45} and
\eqref{2.46} contradicts \eqref{2.249}.
\end{proof}

\begin{proof}[Proof~of~Corollary~\ref{c2.3}]
The assertion follows from Theorems~\ref{t2.13} and \ref{t2.15}.
\end{proof}

\begin{proof}[Proof~of~Corollary~\ref{c2.4}]
The assertion follows from Corollary~\ref{c2.3} and
Lemma~\ref{l2.22} with $\ell_i$ $(i=0,1)$ and $\gamma$ defined by
\eqref{2.201} and \eqref{2.203}, respectively. 
\end{proof}

\begin{proof}[Proof~of~Corollary~\ref{c2.5}]
Note that from \eqref{2.17}, on account of \eqref{2.15}, we have
\eqref{2.20}. Therefore, the assertion follows from
Theorems~\ref{t2.14} and~\ref{t2.16} with $\omega(t)=t-1$ for
$t\in\RR$, and Lemma~\ref{l2.22} with $\ell_i$ $(i=0,1)$ and $\gamma$ defined by
\eqref{2.201} and \eqref{2.203}, respectively.
\end{proof}

\section {Applications}
\label{sec6}

In this section we apply the results obtained above to the model
equations appearing in natural sciences.

  \vspace{0,3cm}
 {\bf{Generalized logistic equation:}} Consider the generalized
 logistic equation
\begin{equation}
\label{6.12}
u'(t)=g_0(t)u(t)\int_{\nu(t)}^t\left|1-\frac{u(s)}{\kappa}\right|^{\lambda}\sgn\left(1-\frac{u(s)}{\kappa}\right)d_sK(t,s),
\end{equation}
where $g_0\in\llocrp$, $\nu:\RR\to\RR$ is a locally essentially bounded
function, $\nu(t)\leq t$ for almost every $t\in\RR$, $\kappa>0$, $\lambda>0$, and
$K:\RR\times\RR\to\RR$ is a measurable function satisfying the following conditions:
\begin{itemize}
\item  $K(t,\cdot):\RR\to\RR$ is a left continuous nondecreasing function of locally
  bounded variation for almost every $t\in\RR$,
\item $\mathcal{K}:\RR\to\RR$, where 
$$
\mathcal{K}(t)\stackrel{def}{=}\int_{\nu(t)}^t d_sK(t,s)\forae t\in\RR,
$$
is an essentially bounded measurable function. 
\end{itemize}

\begin{theo}
\label{t6.2}
Let
$$
\lim_{t\to -\infty}\int_t^{t_0}g_0(s)ds=+\infty,\qquad
\lim_{t\to-\infty}\essinf\big\{\mathcal{K}(s):s\leq t \big\}>0.
$$
Then, for every $t_0\in\RR$ and $c\in(0,\kappa)$ there exists a positive global solution $u$ to
\eqref{6.12} such that 
$$
u(t_0)=c,\qquad u(t)<\kappa \mfor t\leq t_0,\qquad
u'(t)\geq 0 \mforae t\leq t_0,
$$
and there exists a limit $u(-\infty)=0$. 

If, in addition, 
\begin{equation}
\label{6.20}
\lim_{t\to +\infty}\essinf\big\{\nu(s):s\geq t\big\}=+\infty
\end{equation}
and there exists $\omega\in\Sigma$ such that
\begin{equation}
\label{6.14}
\lim_{t\to +\infty}\omega(t)=+\infty,\qquad \limsup_{t\to
  +\infty}\int_{\omega(t)}^t g_0(s)\mathcal{K}(s)ds>0,
\end{equation}
then either $u$ oscillates about $\kappa$ in the neighbourhood of $+\infty$ or there exists a limit $u(+\infty)=\kappa$.
\end{theo} 

\begin{proof}
Let $t_0\in\RR$ and $c\in(0,\kappa)$ be arbitrary but fixed.
Consider the auxiliary equation
\begin{equation}
\label{6.13}
u'(t)=g_0(t)\chi(t,u(t))\int_{\nu(t)}^t\left|1-\frac{u(s)}{\kappa}\right|^{\lambda}\sgn\left(1-\frac{u(s)}{\kappa}\right)d_sK(t,s),
\end{equation}
where 
$$
\chi(t,x)\stackrel{def}{=}
\begin{cases}
(|x|+x)/2 &\casif x< U(t),\\
U(t)&\casif x\geq U(t)
\end{cases}
\for t\in\RR,\quad x\in\RR
$$
and 
$$
U(t)\stackrel{def}{=}
\begin{cases}
\kappa&\casfor t\leq t_0,\\
\kappa\exp\left(\int_{t_0}^tg_0(s)\mathcal{K}(s)ds\right)&\casfor t>t_0. 
\end{cases}
$$
Put $\ell_i\equiv 0$ $(i=0,1)$, $h_0(x)\stackrel{def}{=}x$ for $x\in\RR_+$,
\begin{gather*}
f(v)(t)\stackrel{def}{=}g_0(t)\chi(t,v(t))\int_{\nu(t)}^t\left|1-\frac{v(s)}{\kappa}\right|^{\lambda}\sgn\left(1-\frac{v(s)}{\kappa}\right)d_sK(t,s)\forae
t\in\RR,\\
h_1(v)(t)\stackrel{def}{=}v(t)\int_{\nu(t)}^t\left|1-\frac{v(s)}{\kappa}\right|^{\lambda}\sgn\left(1-\frac{v(s)}{\kappa}\right)d_sK(t,s)\forae
t\in\RR.
\end{gather*}
Then all the assumptions of Theorem~\ref{t2.4} are fulfilled with
$$
q(t,x)\stackrel{def}{=}g_0(t)U(t)\mathcal{K}(t)\mforae t\geq t_0,\quad x\in\RR_+,\qquad
g(t)\stackrel{def}{=}g_0(t)\mathcal{K}(t)\mforae t\in\RR.
$$
Therefore, there exists a global solution $u$ to \eqref{6.13}
satisfying $u(t_0)=c$, $u'(t)\geq 0$ for a.e. $t\leq t_0$, and $0<u(t)<
\kappa$ for $t\leq t_0$. Moreover, also the assumptions of
Corollary~\ref{c2.1} are fulfilled with $\gamma\equiv 1$ and $g\equiv
g_0$. Thus $u(-\infty)=0$.

Now we show that $u$ is positive also on $(t_0,+\infty)$. Assume on the contrary that there
exists $\tau_1>t_0$ such that $u(\tau_1)=0$. Without loss of generality
we can assume that $u(t)>0$ for $t<\tau_1$. Note that $u$ is bounded
on $(-\infty,\tau_1]$, and so there exists $M>0$ such that $u(t)\leq
M$ for $t\leq \tau_1$. Moreover, there exists $\tau_0<\tau_1$ such
that $u(t)< \kappa$ for $t\in[\tau_0,\tau_1]$. Consequently, from
\eqref{6.13} we get
$$
\ln\frac{u(t)}{u(\tau_0)}\geq
-\left|1-\frac{M}{\kappa}\right|^{\lambda}\int_{\tau_0}^t
g_0(s)\mathcal{K}(s)ds\for t\in[\tau_0,\tau_1).
$$
Now the latter inequality yields 
$$
\lim_{t\to\tau_1}\ln\frac{u(t)}{u(\tau_0)}>-\infty
$$   
which contradicts $u(\tau_1)=0$.

Finally we show that $u(t)<U(t)$ for $t\in\RR$ which implies that
$u$ is also a solution to \eqref{6.12}. Assume on the contrary that
there exists $\tau\in\RR$ such that $u(\tau)=U(\tau)$. Obviously,
according to the above proven, $\tau>t_0$ and without loss of
generality we can assume that $u(t)<U(t)$ for $t\in[t_0,\tau)$. Thus
from \eqref{6.13} we get 
\begin{multline*}
u(\tau)=u(t_0)\exp\left(\int_{t_0}^{\tau}g_0(t)
\int_{\nu(t)}^t\left|1-\frac{u(s)}{\kappa}\right|^{\lambda}\sgn\left(1-\frac{u(s)}{\kappa}\right)d_sK(t,s)dt\right)\\
\leq u(t_0)\exp\left(\int_{t_0}^{\tau}g_0(t)\mathcal{K}(t)dt\right)<U(\tau).
\end{multline*}
However, the latter inequality contradicts our assumption. 

Let, in addition, \eqref{6.20} hold and let $\omega\in\Sigma$ be such that \eqref{6.14} is
fulfilled. Then either $u$ oscillates about $\kappa$ in the
neighbourhood of $+\infty$ or there exists
$\tau\in\RR$ such that 
\begin{equation}
\label{6.15}
u(t)\leq\kappa\for t\geq \tau
\end{equation}
or
\begin{equation}
\label{6.16}
u(t)\geq\kappa\for t\geq \tau.
\end{equation}
From \eqref{6.12}, in view of \eqref{6.20}, it follows that $u$ is eventually nondecreasing if
\eqref{6.15} holds and eventually nonincreasing if \eqref{6.16}
is fulfilled. Thus, in both cases there exists a finite limit $u(+\infty)$. 
Therefore, from \eqref{6.12} we get
\begin{equation}
\label{6.17}
\ln\frac{u(t)}{u(\omega(t))}\geq
\left|1-\frac{u(+\infty)}{\kappa}\right|^{\lambda}\int_{\omega(t)}^t
g_0(s)\mathcal{K}(s)ds\for t\geq \tau
\end{equation}
if \eqref{6.15} holds, and 
\begin{equation}
\label{6.18}
\ln\frac{u(t)}{u(\omega(t))}\leq
-\left|1-\frac{u(+\infty)}{\kappa}\right|^{\lambda}\int_{\omega(t)}^t
g_0(s)\mathcal{K}(s)ds\for t\geq \tau
\end{equation}
if \eqref{6.16} is fulfilled. Now both \eqref{6.17} and \eqref{6.18},
in view of \eqref{6.14}, results in
$$
0=\left|1-\frac{u(+\infty)}{\kappa}\right|^{\lambda}.
$$
Consequently, $u(+\infty)=\kappa$.
\end{proof}
 
  \vspace{0,3cm}
 {\bf{ Scalar differential equation without diffusion:}} Consider the
 delay differential equation
\begin{equation}
\label{6.1}
u'(t)=-u(t)+G(u(t-\tau(t)))\for t\in\RR
\end{equation}
where 
\begin{equation}
\label{6.2}
\tau\in\clocrnn,\qquad \limsup_{t\to -\infty}\tau(t)<+\infty,
\end{equation}
and there exists $\kappa>0$ such that the nonlinearity $G$ satisfies
the following conditions:
\begin{gather}
\label{6.3}
G\in\clocrprp,\qquad G(0)=0,\qquad G(s)>s\mfor s\in(0,\kappa),\\
\label{6.4}
\lim_{s\to +\infty}\frac{q_0(s)}{s}=0\qquad\mbox{where}\qquad q_0(s)\stackrel{def}{=}\max\big\{G(x):x\in[0,s]\big\}.
\end{gather}

The delay differential equation \eqref{6.1} covers, e.g., Nicholson's
equation describing the blowflies population, or the Mackey-Glass
equation applied to model white cell production.
As an illustrative example of the function $G$ we can consider 
\begin{equation}
\label{6.5}
G(s)=s^p(\kappa-s)+s\for s\in[0,\kappa],\quad p>0.
\end{equation}

We are interested in the existence of global positive solutions to
\eqref{6.1} satisfying $u(-\infty)=0$. For this purpose let $t_0\in\RR$
and define
\begin{gather}
\label{6.6}
\mu_1(t)\stackrel{def}{=}t\mfor
t\in\RR,\qquad\mu_0(t)=\nu(t)\stackrel{def}{=}t-\tau(t)\mfor
t\in\RR,\\
\label{6.7}
p_0(t)\stackrel{def}{=}
\begin{cases}
1&\casfor t\leq t_0,\\
0&\casfor t>t_0,
\end{cases} 
\qquad p_1(t)\stackrel{def}{=}1\mfor t\in\RR,\\
\label{6.8}
h(t,x,y)\stackrel{def}{=}G(|y|)-p_0(t)y\for t\in\RR,\quad x,y\in\RR,\\
\label{6.9}
q(t,\rho)\stackrel{def}{=}q_0(\rho)\for t\geq t_0,\quad\rho\in\RR_+,
\end{gather}
and consider the problem \eqref{1.3}, \eqref{1.2}. Then it can be
easily verified that all the assumptions of Theorem~\ref{t2.5.r} are
fulfilled. Indeed, we first observe that  $p_i\in\llocrp$,
$\mu_i,\nu:\RR\to \RR$ are locally bounded functions $(i=0,1)$, and
$h:\RR^3\to\RR$ satisfies the Carath\'eodory conditions mentioned in
the introduction, as $G$ is a continuous function. Moreover, the
condition $G(0)=0$ implies that 
$$
h(t,0,0)=0\for t\in\RR,
$$ 
and since $G(s)\geq s$ for $s\in [0,\kappa]$, we have that 
$$
h(t,x,y)\geq 0\for t\in\RR,\quad x,y\in [0,\kappa].
$$
Furthermore, \eqref{6.4} and \eqref{6.7}--\eqref{6.9} implies that 
$$
h(t,x,y)\sgn x\leq q(t,|x|+|y|)\for t>t_0,\quad x,y\in\RR
$$ 
where $q:[t_0,+\infty)\times\RR_+\to\RR_+$ is a Carath\'eodory
function nondecreasing in the second argument and satisfying
\eqref{3.5} for every $b>t_0$. Finally, we also have
$$
\mu_0(t)\leq t,\qquad \mu_1(t)\leq t,\qquad \nu(t)\leq t \for t\in\RR.
$$ 
Thus, the conditions \eqref{2.13}--\eqref{2.15.r} are fulfilled.

On the other hand, observe that $p_i$ satisfy \eqref{2.17} and
\eqref{2.16.r}, as 
$$
\int_{\mu_1(t)}^{t}p_1(s)ds=0\for t\in\RR.
$$
In addition, in view of \eqref{6.2}, \eqref{6.6}, and \eqref{6.7} we have
$$
\int_{\mu_0(t)}^t p_1(s)ds=\tau(t)\leq \sup\big\{\tau(s):s\leq
t_0\big\}<+\infty\for t\leq t_0,
$$ 
and so also the condition \eqref{2.18} is valid. Furthermore,
\eqref{6.3}, \eqref{6.7}, and \eqref{6.8} results in \eqref{2.13.r}. 

Thus, according to Theorem~\ref{t2.5.r} and
Remark~\ref{r2.4.r}, for every $c\in\left(0,\kappa e^{-M_{\tau}}\right)$ with
\begin{equation}
\label{6.10}
M_{\tau}\stackrel{def}{=}\sup\big\{\tau(t):t\leq t_0\big\},
\end{equation}
there exists a global solution $u$ to the problem \eqref{1.3},
\eqref{1.2} having a finite limit $u(-\infty)$ and satisfying
\begin{equation}
\label{6.11}
0\leq u(t)\leq \kappa\mfor t\leq t_0,\qquad u(t)>0 \mfor t>t_0.
\end{equation}
From \eqref{6.6}--\eqref{6.8} and \eqref{6.11} it follows that $u$ is
also a solution to \eqref{6.1}. Moreover, if we put 
$$
h_1(x,y)\stackrel{def}{=}G(y)-y\for x,y\in(0,\kappa)\times(0,\kappa),
$$
then \eqref{2.129} and \eqref{2.41} hold with $g\equiv 1$. Therefore, according to
Corollary~\ref{c2.5} we have $u(-\infty)=0$.
 
Further, let
\begin{gather*}
h_0(y)\stackrel{def}{=}\max\big\{G(s)-s:s\in[0,y]\big\}\for
y\in\RR_+,\\
g(t)=1\for t\leq t_0.
\end{gather*}
Then, obviously, \eqref{2.22} holds. However, \eqref{2.11} is not,
generally speaking, valid, as one can check by the illustrative case
\eqref{6.5}. Obviously, in that case \eqref{2.11} holds if and only if
$p\geq 1$. Consequently, Corollary~\ref{c1.3} can be applied only for
certain $G$ to conclude that $u$ is also positive on the whole real
line.

However, in spite of the fact that the assumptions of
Corollary~\ref{c1.3}, generally speaking, are not fulfilled, still we
can conclude that the solution $u$ is positive on the whole real line
(i.e., also for $p\in(0,1)$ provided \eqref{6.5} is
fulfilled). Indeed, the positivity of $u$ is guaranteed by the
following assertion.

\begin{lemma}
\label{l6.1}
Let \eqref{6.2} and \eqref{6.3} hold. If $u$ is a nontrivial non-negative solution to \eqref{6.1}, then $u(t)>0$ for $t\in\RR$.
\end{lemma}

\begin{proof}
Suppose on the contrary that there exists $\eta\in\RR$ such that $u(\eta)=0$. Then we have
$$
u(t)=-\int_t^{\eta} e^{s-t}G(u(s-\tau(s)))ds\leq 0\for t\leq \eta.
$$
Since $u\geq 0$ for $t\in\RR$, we can conclude that $u(t)=0$ for
$t\leq \eta$. In addition, since $u$ is a nontrivial non-negative
function, there exists $\zeta\in\RR$ such that
$u(\zeta)>0$. Obviously, $\eta<\zeta$ and without loss of generality
we can assume that $u(t)>0$ for $t\in(\eta,\zeta]$.  

Since $\tau(t)>0$ for $t\in\RR$ is continuous, there exists $\varepsilon>0$ such
that $t-\tau(t)\leq \eta$ for $t\in[\eta,\eta+\varepsilon]$, and hence
$u(t-\tau(t))=0$ for $t\in[\eta,\eta+\varepsilon]$. Since $G(0)=0$, we
have $G(u(t-\tau(t)))=0$ for $t\in
[\eta,\eta+\varepsilon]$. Consequently, from \eqref{6.1} it follows
that 
$$
u'(t)=-u(t)\for t\in[\eta,\eta+\varepsilon],\qquad u(\eta)=0,
$$
whence we get $u(t)=0$ for $t\in [\eta,\eta+\varepsilon]$, a contradiction. 
\end{proof}
 
Therefore, the above-mentioned discussion and Lemma~\ref{l6.1} results
in the following assertion.

\begin{theo} 
\label{t6.1}
Let \eqref{6.2}--\eqref{6.4} hold. Then, for each $t_0\in \RR$ and
$c\in \left(0,\kappa e^{-M_{\tau}}\right]$ with $M_{\tau}$ given by
\eqref{6.10}, there exists a positive global solution $u$ to
\eqref{6.1} such that 
$$
u(t_0)=c,\qquad  u(t)\leq \kappa \mfor t\leq t_0,
$$
and there exists a limit $u(-\infty)=0$.  
\end{theo}

\begin{rem}
\label{r6.1}
In spite of Theorem~\ref{t2.5.r}, the value $c=\kappa e^{-M_\tau}$ is
admissible in Theorem~\ref{t6.1}, because the function $\tau$ is
continuous and $p_1(t)>0$ for $t\in\RR$. Consequently, a function
$\sigma$ can be directly defined as $\sigma(t)=t-\tau(t)$ for
$t\in\RR$ (see the proof of Theorem~\ref{t2.5} for more details).  
\end{rem}

\begin{rem} 
\label{r6.2}
Note that the typical condition on $G$: ``$G$ is differentiable at 0''
is not used in the proof of Theorem~\ref{t6.1}. Therefore, the
results presented complete or improve the already known results.
\end{rem}

\end{document}